\documentclass{amsart}
\usepackage[utf8]{inputenc}
\usepackage[letterpaper,margin=1in]{geometry}
\usepackage{hyperref}
\usepackage{amsthm}
\usepackage{graphicx}
\usepackage{amsmath}
\usepackage{amssymb}
\usepackage{mathrsfs}
\usepackage{multicol}
\usepackage{mathtools}
\usepackage{hyperref}
\usepackage{hhline}
\usepackage{subcaption}
\usepackage{changepage}
\usepackage[english]{babel}
\usepackage[dvipsnames]{xcolor}
\usepackage{appendix}
\usepackage{pdflscape}
\usepackage{rotating}
\usepackage[ruled,vlined]{algorithm2e}

\hypersetup{
    colorlinks=true,
    linkcolor=blue,
    filecolor=magenta,
    urlcolor=blue,
    citecolor=blue
}
 
\makeatletter
\newtheorem*{rep@theorem}{\rep@title}
\newcommand{\newreptheorem}[2]{%
\newenvironment{rep#1}[1]{%
 \def\rep@title{#2 \ref{##1}}%
 \begin{rep@theorem}}%
 {\end{rep@theorem}}}
\makeatother

\newtheorem{theorem}{Theorem}[section]
\newtheorem{remark}[theorem]{Remark}
\newtheorem{lemma}[theorem]{Lemma}
\usepackage[foot]{amsaddr}

\newcommand{\allT}{T}
\newcommand{\allD}{\xi}
\newcommand{\pop}{\text{pop}}

\newcommand{\tree}{\tau}

\usepackage{tikz}

\tikzset{
    dot/.style 2 args={fill, circle, inner sep=2pt, label={#1:\scriptsize #2}}
}

\date{\today}
\begin{document}

\title[Multi-Scale Merge-Split Markov Chain Monte Carlo for Redistricting]{Multi-Scale Merge-Split Markov Chain Monte Carlo\\ for Redistricting}

\author{Eric A. Autry}

\author{Daniel Carter}
\address[Daniel Carter, Zach Hunter]{North Carolina School of Science and Mathematics, Durham NC}

\author{Gregory Herschlag}
\address[Eric A. Autry,Gregory Herschlag, Jonathan Mattingly]{Department of Mathematics, Duke University}
\email{gjh@math.duke.edu}

\author{Zach Hunter}

\author{Jonathan C. Mattingly}
\address[Jonathan Mattingly]{Department of Statistical Science, Duke University}
\email{jonm@math.duke.edu}

\begin{abstract}
  We develop a Multi-Scale Merge-Split Markov chain on redistricting plans. The chain is designed to be usable as the proposal in a Markov Chain Monte Carlo (MCMC) algorithm. Sampling the space of plans amounts to dividing a graph into a partition with a specified number of elements which each correspond to a different district. The districts satisfy a collection of hard constraints and the measure may be weighted with regard to a number of other criteria. The multi-scale algorithm is similar to our previously developed Merge-Split proposal, however, this algorithm provides improved scaling properties and may also be used to preserve nested communities of interest such as counties and precincts. Both works use a proposal which extends the ReCom algorithm which leveraged spanning trees merge and split districts. In this work we extend the state space so that each district is defined by a hierarchy of trees. In this sense, the proposal step in both algorithms can be seen as a ``Forest ReCom.'' We also expand the state space to include edges that link specified districts, which further improves the computational efficiency of our algorithm. The collection of plans sampled by the MCMC algorithm can serve as a baseline against which a particular plan of interest is compared. If a given plan has different racial or partisan qualities than what is typical of the collection of plans, the given plan may have been gerrymandered and is labeled as an outlier.  
\end{abstract}

\maketitle

Comparing a given redistricting plan to an ensemble of neutrally drawn plans is quickly becoming a standard method for identifying partisan and racial gerrymanders.  An ensemble of plans serves as a baseline against which a particular plan of interest is compared. If the given plan has different racial or partisan qualities than what is typical of the collection plans, the given plan may have been gerrymandered and labeled as an outlier. \footnote{For more discussion see \cite{QuantifyingGerrymanderingBlog}.}  This approach has been used by elected officials when considering remedial maps \cite{duchinPAreport} and has been successfully employed as evidence in a number of recent court cases \cite{RWCAvWakeBOE,GreensboroVGuilford,CovingtonvNC,GillVWhitford,LWVvPA,RuchoVCC,LewisVCommonCause}.  Methods for generating the ensembles are varied: There are constructive randomized algorithms including seed and flood and assimilation methods \cite{cirincione2000assessing,ChenRodden13,Chen15}, optimization algorithms \cite{mehrotra1998optimization,Liu16}, moving boundary MCMC algorithms \cite{macmillan2001redistricting,MattinglyVaughn2014,QuantifyingGerrymandering,Wu15,fifield2015,jcmReport}, local chain comparison algorithms \cite{Chikina_Frieze_Pegden_2017,chikina2019separating}, and also tree based methods that rearranges pairs of districts by cutting spanning trees \cite{moonVa, deford2019redistricting,DeFord2018,DeFordDuchinPrivite,deford2019recombination,carter2019optimal,CannonConference2020}.  All of the tree based methods build on the initial work of ReCom \cite{deford2019redistricting}.

It is conjectured that tree based methods may be able to mix quickly due the relatively global changes they make; in contrast methods such as moving boundary or single node flip methods \cite{MattinglyVaughn2014, fifield2015} may suffer from the fact that paths between acceptable redistricting plans may have to pass large energetic or entropic barriers \cite{njatDedfordSolomon2019graphs}. Tree based methods have also been made reversible while sampling from a given measure making them a promising path in sampling problems on redistricting plans \cite{carter2019optimal,CannonConference2020}.

In general, a redistricting plan may be thought of as a maps from nodes on a graph to an assigned district; the nodes of the graph represent political regions such as census blocks, precincts or counties within a region of interest. Due to common political criteria, the induced subgraph of the nodes assigned to a district is simply connected, meaning that we may think of the districting plan as a partition of the graph.

It is, however, at times useful to extend the state space to a higher dimensional space. For example, the primary idea in \cite{carter2019merge} is to expand state space from a partition to be a spanning forest. Such an extension makes computing reversibility feasible (see e.g. \cite{deford2019recombination}).

Despite the benefits of tree based algorithms, such methods have not yet been shown to be able to preserve larger geographical units. Many state redistricting guidelines call for the preservation of communities of interest. These communities include counties, precincts, municipalities, Native American reservations and neighborhoods. We have previously developed techniques that provide soft energetic constraints on these elements \cite{MattinglyVaughn2014,Bangia17,herschlag2017evaluating,Herschlag20,jcmReport}. These methods rely on small changes at the boundary that accumulate into larger scale moves. On large scale problems, these methods do not always sample from a given distribution and also may not provide strict control over the number of split counties: As an example, the 2016 NC congressional remedial redistricting plan split 13 counties, and the 2019 plan split 12;\footnote{Given the extremely tight population constraints on congressional districts, it is reasonable to assume that there is no subset of counties that perfectly can accomodate a subset of the congressional districts. Under this assumption our previous work may be used to demonstrate that 12 county splits is optimal \cite{carter2019optimal}.} when using soft constraints, however, we previously constructed an ensemble that split, in median, 34 counties \cite{Herschlag20}.

In addition, tree based methods have polynomial complexity for each proposal step due to the need to use Kirchoff's theorem to compute the number of trees on a graph. Perhaps due to this complexity, there has been little work on sampling the space of redistricting plans through Markov chains that extends to the level of census blocks. 

In the current work, we employ a hierarchy of quotient graphs (i.e. a hierarchy of nested partitions) in a multi-scale framework. We will demonstrate the possibility for logarithmic, rather than polynomial, complexity which promises to yield samples of redistricting plans fully resolved at the finest levels.
This reduction is primarily due to employing Kirchoff's theorem on fixed size partitions, because as the size of the graph grows one may simply increase the number of levels in the hierarchy (see Section~\ref{sssec:complexHT}).
This hierarchy may reflect higher level redistricting features such as precincts and counties, or alternatively, the hierarchical structure may itself evolve and be decoupled from political boundaries. Nodes within each level of the hierarchy are mostly preserved, meaning that we may preserve such higher level features. By coupling the hierarchical space to our previous work in \cite{carter2019merge}, we are able to make large changes to district boundaries at each step with a scheme that is both able to preserve geographic regions of interest and is computationally efficient.

Fundamentally, the multi-scale framework is equivalent to placing a novel measure the spanning forest while modifying proposal matrix to focus on plans that preserve hierarchical structures (and that may, on occasion, update the hierarchical structures).
In addition to the algorithmic advances provided by the multi-scale framework, we further expand the state space to consider edges between districts. Similar to the expansion to a spanning forest, such edges will greatly simplify computing the acceptance ratio when applying the Metropolis-Hastings algorithm to our proposal chain.

In developing the algorithm, we demonstrate its capabilities on the precincts embedded within the counties of North Carolina on the 13 congressional districts.  We implement a method that splits 13 or fewer counties and find evidence of convergence on observables of interest across 10 independent and randomly seeded Markov chains.

\subsection*{Overview: } In the next section we give a high level overview of the Multi-Scale Merge-Split Algorithm. In Section~\ref{sec:TargetMeaasure}, we describe the family of target measure which we use the Multi-Scale Merge-Split proposal to sample from.
We also describe the general setting in which we will work and define the space of Hierarchical Trees and the idea of Linked edges, both of which are central to the paper. In Section~\ref{sssec:complexHT}, we discuss the computational complexity of the algorithm. In Section~\ref{Sec:StructreuMeasure}, we collect some observations on the structure of the target measure which will be useful in investigating the algorithm. In Section~\ref{sec:samplingFromP}, we describe the general Metropolis-Hastings scheme used in our algorithm. In Section~\ref{sec:mergeSplitQ}, we layout our proposal algorithm, step-by-step, giving some details of the implementation. More details on the implementation are given in Section~\ref{sec:implement}. Section~\ref{sec:NumericalResults} we give some initial numerical results. We close the paper with a discussion and two appendices which contain some details on how to evolve the multi-scale hierarchies dynamically and some computational issues the bookkeeping of linked edges.

\section{Informal Overview of the Reversible Multi-Scale Merge-Split Algorithm}
We describe the algorithm in the context of political redistricting as that is our main application of interest. However, at heart, the algorithm is a multi-level/multi-scale graph partitioning algorithm.
Typically, political districts are formed out of atomic geographic elements such census blocks, precincts, or counties. A redistricting is simply an assignment of each of these smallest atomic elements to a district (see Figure~\ref{sfig:overviewAssignment}). Since districts are typically required to be connected, we define each district as a connected partition of the atomic elements. 

In many redistricting applications, certain higher level features, like counties, are preserved when possible. Each atomic unit may belong to a number of higher level features: For example, a census block is contained within a precinct which is contained within a county. Often these descriptive levels form nested hierarchies (e.g. precincts are contained within counties). When describing a districting graph, officials will often list nodes corresponding to the coarsest descriptions that are fully assigned to a single district, and recursively list finer level district assignments (see Figure~\ref{sfig:overviewMixedAssignment} and e.g. \cite{HB1029}). 

Alternatively, one can invent nested hierarchies in order to induce a multi-scale, and compressed, districting description. These invented hierarchies can evolve as part of the algorithm (see Section~\ref{apdx:DynamicHierarchies}). In this note, we primarily describe the case where the hierarchy is kept constant, since evolving the hierarchy is a layer which can be built on top of a fixed hierarchy algorithm as described in this paper.

\begin{figure}
\centering
\subcaptionbox{Precinct Assignment\label{sfig:overviewAssignment}}{\includegraphics[width=0.45\linewidth, clip = true, trim = {0cm 5cm 0cm 3cm}]{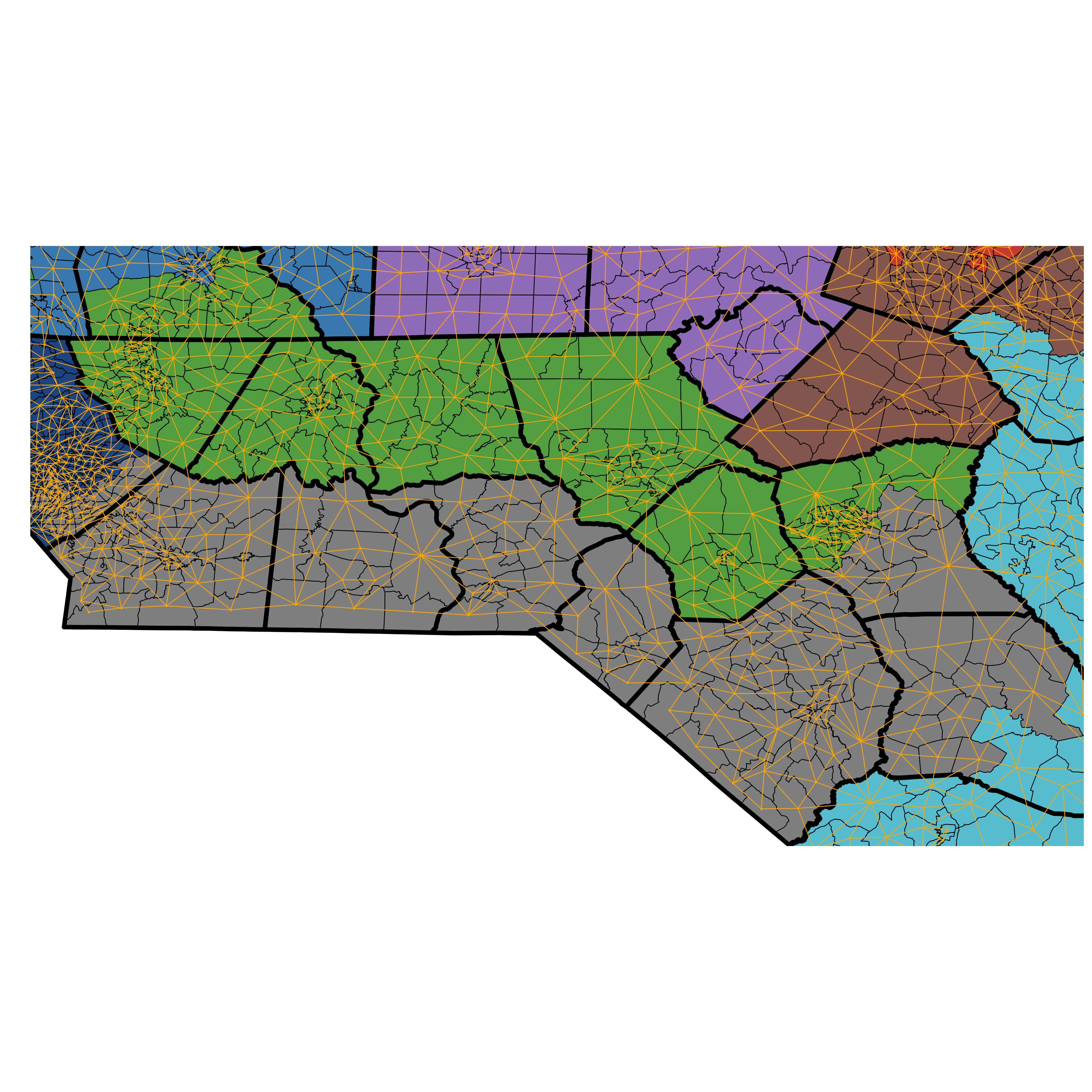}}\qquad
\subcaptionbox{Mixed Level Assignment\label{sfig:overviewMixedAssignment}}{\includegraphics[width=0.45\linewidth, clip = true, trim = {0cm 5cm 0cm 3cm}]{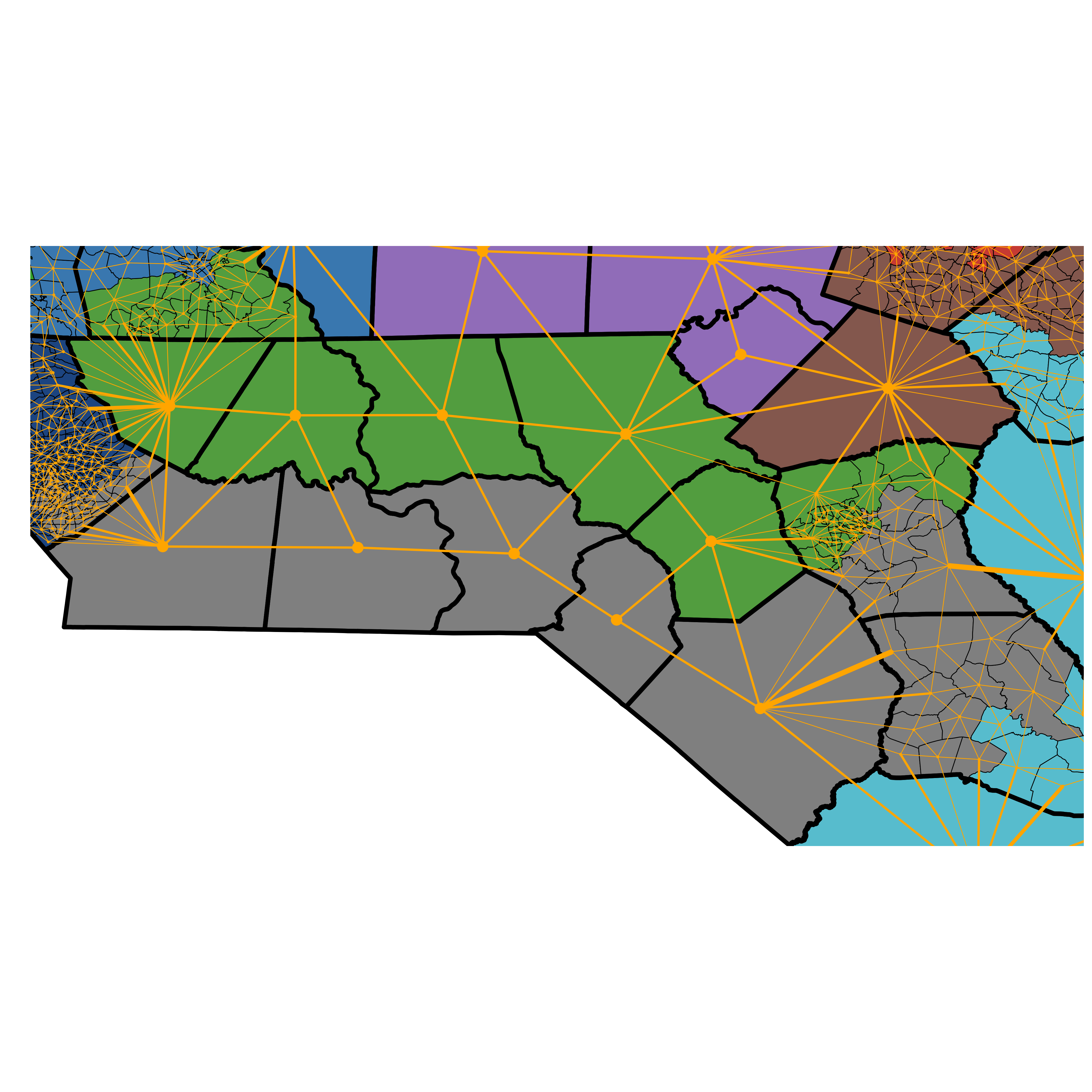}}\qquad
\caption{We display a district assignment at the precinct level (A); counties are shown with thick black lines. The precinct graph is shown in orange. We also show a mixed-level description of the district assignment (B); counties that are preserved are collapsed to single nodes and edges are weighted (shown by thickness) based on how many adjacent precincts link two nodes.}
\label{fig:overviewAssignment}
\end{figure}

\begin{figure}
\centering
\subcaptionbox{New top-level tree on merged districts \label{sfig:topLevelTree}}{\includegraphics[width=0.45\linewidth, clip = true, trim = {0cm 5cm 0cm 3cm}]{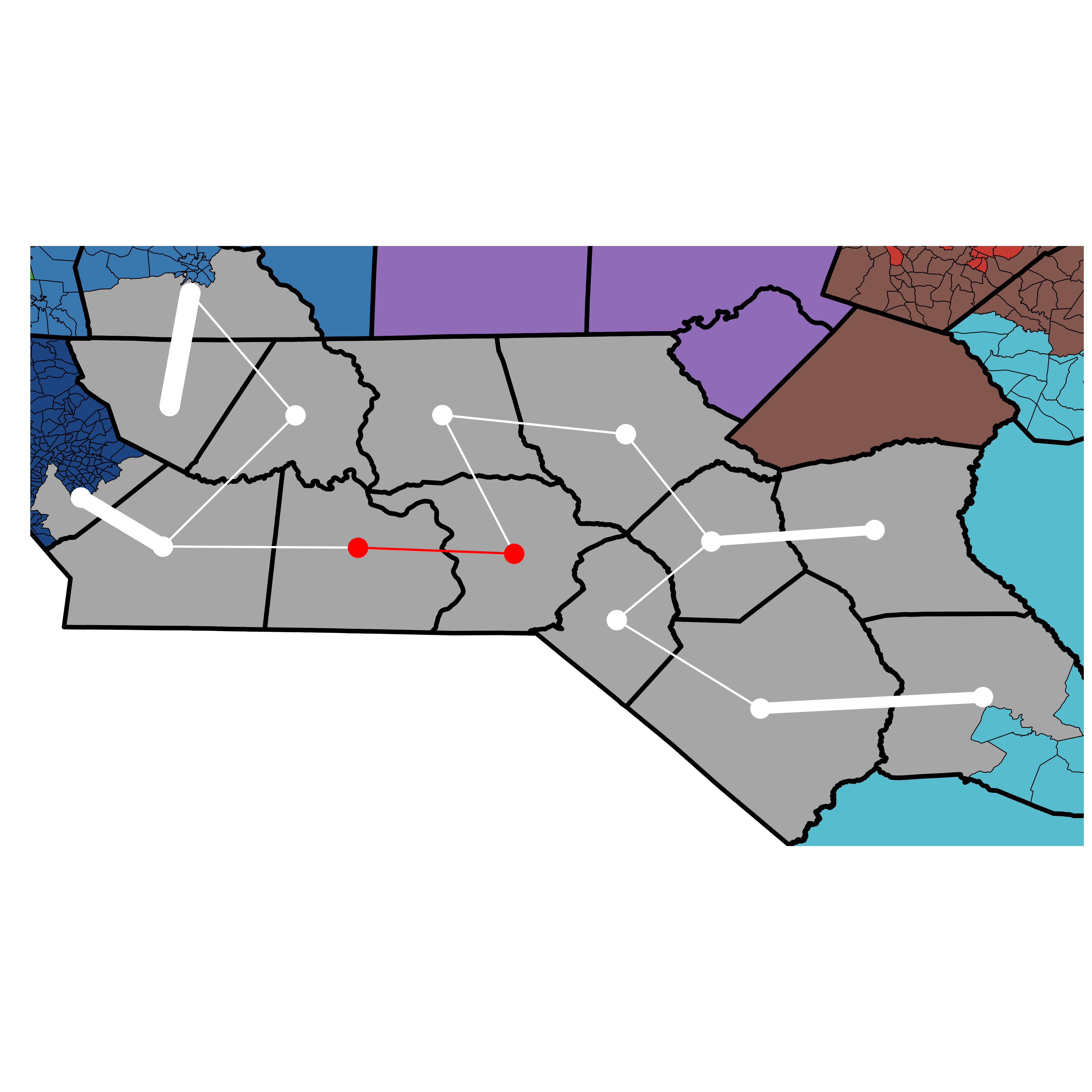}}\qquad
\subcaptionbox{Expanded multi-level tree \label{sfig:expandedTree}}{\includegraphics[width=0.45\linewidth, clip = true, trim = {0cm 5cm 0cm 3cm}]{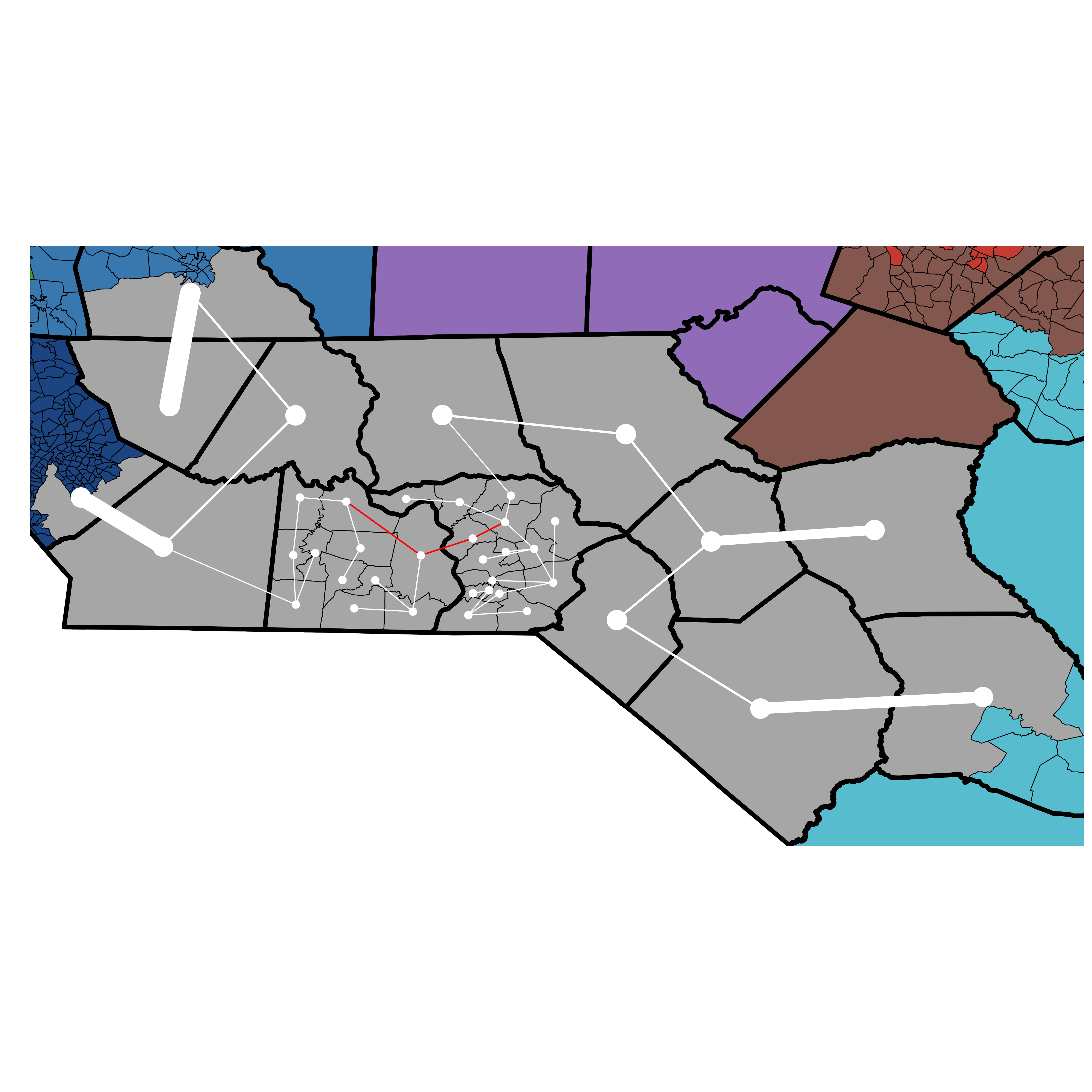}}\qquad
\subcaptionbox{New forest \label{sfig:newForest}}{\includegraphics[width=0.45\linewidth, clip = true, trim = {0cm 5cm 0cm 3cm}]{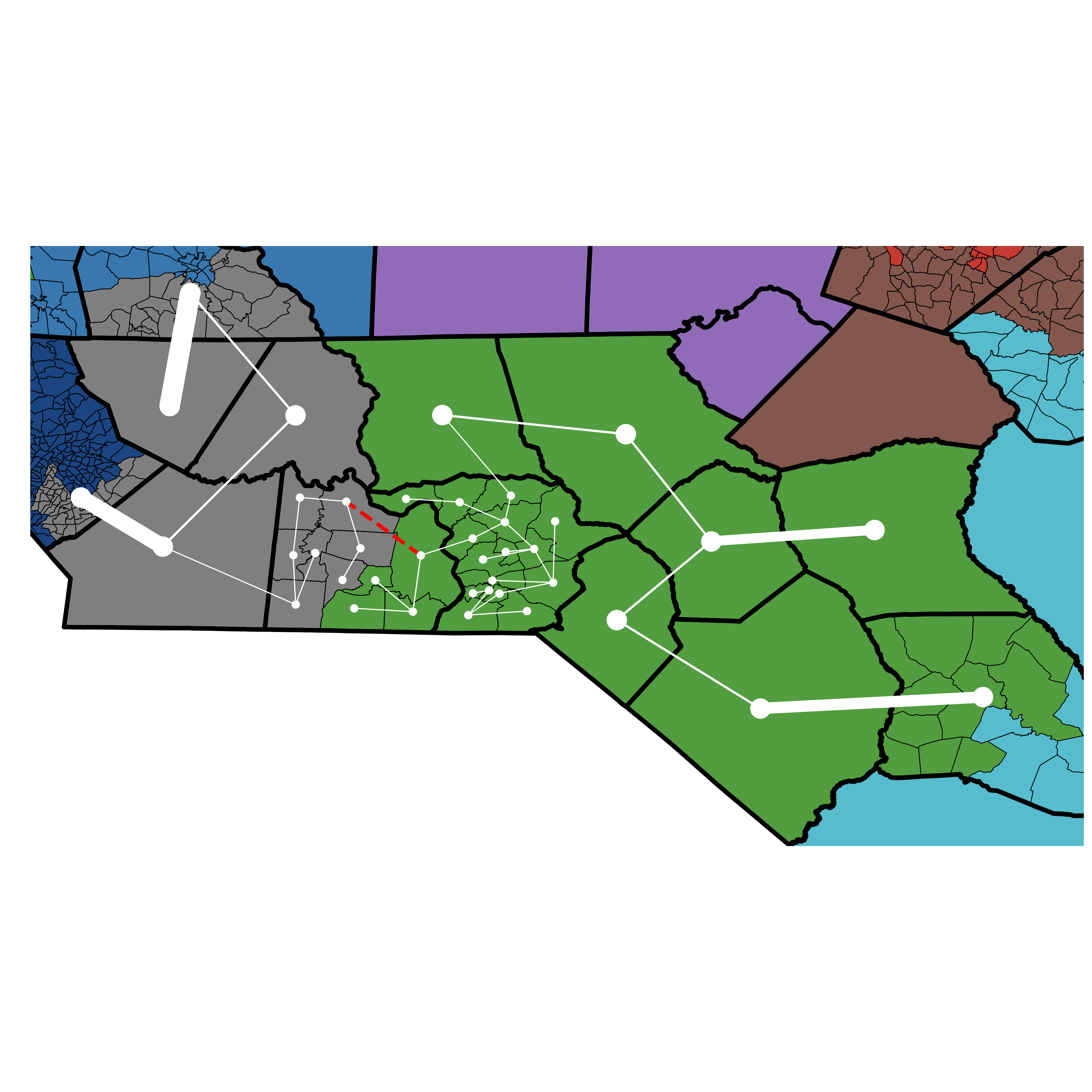}}\qquad
\caption{We display a merged district and a possible spanning tree at the top level (A); this tree shows both nodes and edges that may be cut in red. The nodes that might be cut are expanded and the edges across counties are randomly specified according to the edge weights (B); in the expanded graph we show the edges that may be cut to maintain acceptable population parity. A random edge is cut and the two new districts are specified (C); the trees persist as part of the new state, as does the cut edge.}
\label{fig:overviewProposal}
\end{figure}

We represent the nested features as graphs, with nodes corresponding to geographical units, and edges corresponding to adjacent units (see Figure~\ref{sfig:overviewAssignment}). When examining a graph with preserved larger scale features, we quotient out (or coarsen) these larger scale features in the graph (see Figure~\ref{sfig:overviewMixedAssignment}). We weight the edges between coarse nodes based on how many edges at the finest level span the coarser nodes.

We will use a tree based algorithm to sample the space of plans.  Tree based sampling algorithms for districting plans began with the ReCom \cite{deford2019recombination} algorithm which merges two adjacent districts, draws a spanning tree on the merged districts, and looks for edges that, if removed, would yield two sub-trees that each contain an acceptable population. The Merge-Split algorithm \cite{carter2019merge} and reversible ReCom algorithm \cite{CannonConference2020} build on this idea in different ways with the common goal of making explicit the invariant measure that is being sampled. The distinguishing feature of the Merge-Split algorithm is that it ustulizes a proposal Markov chain on a collection of trees, each spanning a district, in a Metropolis-Hastings scheme. In this way, the proposal step of Merge-Split might be reasonably called ``forest ReCom.''  In contrast, reversible ReCom sets a maximum number of allowable cut edges and evolves on a state space in which the districts are defined as a partition of the graph.  All three algorithms calculate spanning trees on merged districts, but ReCom algorithms discard the tree after splitting. In Merge-Split the spanning forest is retained which provides us with an efficient method to compute the forward and backward transition probabilities in a Metropolis-Hastings scheme. This allows us to sample from a specified target distribution which is often important in practical settings (reversible ReCom also has this property). These algorithms can also be mixed with other sampling techniques in efforts to improve exploration of the space and mixing.

In this work, we keep the spanning forest as the state space of the Markov Chain; we will further expand the state space to further simplify the calculation of the forward and backward probability. Although the primary idea of the Multi-Scale Merge-Split algorithm is similar to the Merge-Split algorithm, we first draw a tree at the coarsest level of the quotient graph established by the induced subgraph of the two merged districts (see Figure~\ref{sfig:topLevelTree}). We then look for edges that could be removed or nodes that could be ``cut'' to lead to acceptable population differences between the resulting trees. For the nodes that can be cut, we expand the graph at the finer level, draw a tree within the expanded node and again look for removable edges or cuttable nodes (see Figure~\ref{sfig:expandedTree}). 
We only reveal finer scale structure, as needed.
When the finer scale is needed it can be drawn since its distribution is determined by local information on the coarse scales above.  
Such as expanding a coarse edge or node.
Once all cuttable nodes have been expanded, we have a list of edges that can be removed while preserving acceptable population balance, and we choose one to create the new districts (see Figure~\ref{sfig:newForest}).

Merging and splitting will result in a new spanning forest where two the the trees have been updated. We will see below how this allows us to computationally track the acceptance ratio when applying the Metropolis-Hastings algorithm to the proposal. ``Coarse nodes, even those contained within split counties, do not immediately need to be specified upon a Merge-Split step (see Figure~\ref{sfig:newForest}). Instead we will assume that there is a random tree within each coarse node and a random edge chosen where nodes are coarser than the finest level. One of the benefits of this method is that we do not need to immediately specify which tree or edge is chosen, only the distribution of unspecified edges and trees. This means that we can maintain a multi-scale representation of the tree that will save computational time (i) in sampling new trees, (ii) in counting the number of trees, and (iii) in allowing for a compressed description of the districting graph when writing to disc. 

In addition to keeping the hierarchical forest, we will further expand the state space by keeping track of which edge we removed (shown as the dashed red line in Figure~\ref{sfig:newForest}); we will call these edges \emph{linking edges}. We will see how keeping track of such edges will further simplify the computational complexity. This has two effects. It can reduce the amount of nodes that need to be expanded to finer scales thereby reducing the computation complexity; it also further reduces the calculation of the forward and backward probability. The state space is then defined to be a hierarchical spanning forest combined with a set of linking edges.

\section{The Setting and Target Measure}\label{sec:TargetMeaasure}
We will now lay the groundwork needed to more formally define the algorithm sketched in the previous section.
Let a graph $G_0$ have vertices $V_0$ and edges $E_0$. Each vertex will represent some geographic region to be assigned to a district plan -- this could be a voter tabulation district (VTD), precinct, census block, county, etc. In this context, edges are placed between vertices that are either rook, queen, or legally adjacent.\footnote{Rook adjacency means that the geographical boundary between two regions has non-zero length; queen adjacency means that the boundaries touch, but may do so at a point. At times two regions may not be geographically adjacent, but may be considered adjacent for legal purposes; for example, an island may still be considered adjacent to regions on a mainland for the purposes of making districts.}
Furthermore, in this context, we will be working with planar graphs, though all of the ideas we will discuss may be trivially expanded to generic graphs.

We may always represent a districting plan on $G_0$, made up of $d$ districts, as a function $\xi:V \to \{1,2 \dots d\}$. Informally, $\xi(v)=i$ means $v$ is in the $i$th district. Given a districting plan $\xi$, we will denote by 
$V_i(\xi) = \{v \in V \mid \xi(v) = i\}$ and $E_i(\xi) = \{(v,u) \in E \mid \xi(v) = \xi(u) = i\}$ as the set of vertices in the $i$th district and the set of edges between vertices in the $i$th district, respectively. We will define $\xi_i = (V_i(\xi), E_i(\xi))$ to be the subgraph induced by the $i$th district. More generally, given a graph $G = (V_G, E_G)$ with vertices $V_G$ and edges $E_G$, we will define $V(G) = V_G$ and $E(G) = E_G$ as maps from a graph to a vertex or edge set, respectively. 

We will also sometimes associate extra data with the vertices and edges, such as population, land area, and border length. The additional data is used to evaluate the districts on desired redistricting criteria, such as equal-population and compactness. Of particular note, we define $\pop(v)$ to be the population of vertex $v$ and
\begin{align}
\pop(\xi_i) = \sum_{v\in V_i(\xi)} \pop(v).
\end{align}
to be the population of district $\xi_i$. Each vertex of $G_0$ may also be associated with a graph partition. This partition may either represent arbitrary sets of vertices or represent larger geographical units. For example, a vertex representing a census block has corresponding data of the precinct and county in which it resides.

As in the original Merge-Split algorithm \cite{carter2019optimal}, we begin by expanding our state space to be the set of $d$-\textit{tree partitions} of the atomic graph, i.e. the space of forests on $G_0$ consisting of $d$ disjoint trees. We will use the term  \textit{spanning forest} interchangeably for such a collection of disjoint trees which span the graph. From this perspective the state space has elements of the form
\begin{align*}
 \allT = \{T_1,T_2,\cdots,T_d\},
\end{align*}
where each $T_i$ is a spanning tree on the subgraph $\xi_i$ with vertices $v_i=V_i(\xi)$ and edges $\varepsilon_i\subseteq E_i(\xi)$. Unlike the original Merge-Split algorithm, we will limit the space of allowable trees to a subset of the spanning trees that we will denote as \emph{hierarchical trees}. Hierarchical trees have the property that when they are projected to a collection coarser scales, they spanning trees. Limiting to the class of hierarchical trees will allow us both to preserve hierarchical partitions as well as provide favorable scaling the computational complexity of the algorithm. In the preserved hierarchical partitions, it will often suffice to know that there is some more finely resolved tree without knowing exactly what this tree is; therefore we will be able to coarsen the hierarchical tree in certain regions leading to a tree resolved over multiple scales.

The restricted tree space arises through a nested hierarchy of partitions. We will denote $H_0(G) = G$ as the the finest level graph in a hierarchy of graphs that we will construct inductively based on the nested partitions. We begin with a partition on $H_0$, denoted $\mathcal{P}_0$, and induce a coarser graph, $H_1$, which has one vertex corresponding to each element of the partition and edges for each edge on $H_0$ that connects two vertices in different partitions. In particular, we allow for multiple edges between vertices in $H_1$.  We then choose a partition of $H_1$, denoted $\mathcal{P}_1$, and repeat the process to obtain $H_2$. We continue iteratively until we have reached our desired top level, which we will denote $H_\ell$.
We will denote the hierarchy of graphs, $\{H_i\}_{i=0}^\ell$, as $\mathcal{H}$.  

At times we will want to apply this coarsening procedure to a subgraph of $H_i$. For this reason, we define a coarsening operator, $\mathcal{Q}$ which maps a finer graph to a coarser graph in the hierarchy; for example, for $H_i'\subset H_i$ $\mathcal{Q}^{j}(H_i') = H_{i+j}'\subset H_{i+j}$. Similarly, we define an expansion operator as the inverse of $\mathcal{Q}$, $\mathcal{Q}^{-1}$.  In Section~\ref{apdx:DynamicHierarchies}, we use multiple hierarchies of partitions that lead to different coarsening ladders of the base graph $G$.  If we have two hierarchies, $\mathcal{H}$ and $\mathcal{H}'$, each induces its own coarsening operator and we will distinguish them by $\mathcal{Q}_\mathcal{H}$ and $\mathcal{Q}_\mathcal{H'}$.  Moving to different hierarchies allows us to dynamically change our coarse units which is useful when they are not relevant to the redistricting problem and only introduced for computational efficiency.

Formally, when computing the number of edges between two coarse nodes at a given level, $u, v\in V(H_n)$, we let $w_n(u, v)$ be the number of edges linking the two nodes $u, v$.
We calculate the number of edges recursively by first setting 
\begin{align*}
  w_0(u, v)=
  \begin{cases}
    1 & \text{ if } (u,v) \in E(H_0)\\
   0 &\text{ if } (u,v) \not\in E(H_0)
  \end{cases}
\end{align*}
and then defining $w_n(u, v)$ by
\begin{align}
\label{eq:mgweight}
w_n(u, v) &= \sum_{(u', v')\in \mathcal{E}_n(u, v)} w_{n-1}(u', v')
\end{align}
where
\begin{align*}
\mathcal{E}_n(u, v)= \Big\{(u', v')\in E(H_{n-1}) \: | V(\mathcal{Q}(u')) = u \text{ and } V(\mathcal{Q}(v')) = v \Big\}\,.
\end{align*}

When considering the space of allowable trees on the induced subgraph of district $\xi_i$, we require that any of the recursive quotient graphs on $T_i$ be trees themselves with no more than a single edge connecting nodes in the recursive multi-graphs; this is to say that there is, at most, only one edge at the base level that spans partitions at any given level. Formally, we define $T_i$ to be a \emph{hierarchical tree} on $\mathcal{H}(\xi_i)$ which lives in the space of hierarchical trees, $HT$, defined as
\begin{align}
HT(\mathcal{H}(\xi_i)) \equiv \{t \in ST(\xi_i) \: | \: \forall n, \mathcal{Q}^n(t)\in ST(H_n(\xi_i))\},
\end{align}
where $ST(\cdot)$ is the set of spanning trees on the argument graph. We remark that for $t\in HT(\xi_i)$, and for all $(u,v)\in E(\mathcal{Q}^n(t))$, we have that $w_n(u,v) = 1$, since otherwise there would be a two node loop in the multi-graph on $\mathcal{Q}^n(t)$ which would contradict the requirement that $\mathcal{Q}^n(t) \in ST(H_n(\xi_i))$.

As shown in Figure~\ref{sfig:newForest} we will be looking for edges on hierarchical trees over merged districts that could be removed to yield two districts with acceptable populations. Previously, we did not keep track of the edges we removed when splitting merged trees to obtain two new spanning trees \cite{carter2019merge}, however in the current work we will see that is is useful to track these edges for computational ease. These are what we previously referred to as the linking edges, as they are edges that linked two trees when they were last split. We denote the set of linking edges $L$; note $(u,v)\in L$ implies that $\xi(u)\neq\xi(v)$. We also require that any two districts have at most one edge that links them. The linking edges further expanded state space to be 
\begin{align*}
 \allT \times L,
\end{align*}
or a combination of a hierarchical forest together with a collection of linking edges. 

\subsection{The target measure on hierarchical forests}
We will now place the probability measure on the expanded state space consisting of $d$ disjoint trees $\allT = \{T_1,T_2,\cdots,T_d\}$ and a set of edges linking some of these trees, $L$. We take our measure to be of the form
\begin{align}\label{eq:Pdef}
P(\allT, L) \propto e^{-\beta J(\xi(\allT))} \big(\tau_\mathcal{H}(\xi(\allT))\times \mathcal{L}(\xi(\allT))\big)^{-\gamma},
\end{align}
where $\xi(\allT)$ is the districting plan induced by the forest $\allT$, $J$ is a score function that evaluates how ``good'' a districting plan is,\footnote{Lower scores are ``better'' in the sense that a districting plan in question performs better when considering criteria included in the definition of $J$. } and $\beta\in[0,1]$ and $\gamma\in[0,1]$ are tempering parameters used to change the importance of the factors $J(\xi)$ and $(\tau_\mathcal{H}(\xi) \times \mathcal{L}(\xi))$, respectively. $\tau_\mathcal{H}(\xi(\allT))$ is a count of the number of hierarchical forests corresponding to $\xi(\allT)$ and $\mathcal{L}$ is the number of allowable linking edge arrangements corresponding to $\xi(\allT)$. 

$\tau_\mathcal{H}(\xi(\allT))$ can be represented as
\begin{align}
\tau_\mathcal{H}(\xi(\allT)) \equiv \prod_{i = 1}^d \tau_\mathcal{H}(\xi_i),
\end{align}
since the choice of tree in each partition is independent.

There are several choices for what linking edge sets are permissible and $\mathcal{L}$ will depend on this choice.  Below in Section~\ref{ssec:linkingEdgeSets} we describe two possible choices.  As an example, we may consider tracking a linking edge across all adjacent districts such that the linking edge is contained within a node at finest level that shares the two districts; if the districts do not share a node at any level then they are linked across the coarsest nodes.  In this case, we would count the number of possible linking edges for all adjacent districts and take the product of these independent choices to count the number of ways to assign linking edges to this partition.  In general, computing the number of allowable linking edge sets for a given partition $\xi$ typically reduces to the product over a (partial) boundary count over (possibly a subset of) adjacent districts.

\subsubsection{Counting hierarchical trees}
\label{sssec:countComplexHT}
We show how to count the number of hierarchical trees.  Below in Section~\ref{sssec:complexHT}, we show that using the multi-scale setting with pre-computation may lead to reducing the complexity of counting these trees from polynomial to logorithmic time.

\begin{theorem}[Hierarchical tree count]
\label{thrm:HTcount}
Given a graph $G$ with graph hierarchy $\mathcal{H}(G)$, the number of hierarchical trees in $HT(G)$ may be computed as
\begin{align}
\tau_{\mathcal{H}}(G) &= \tau(H_\ell(G)) \prod_{n = 1}^\ell \Bigg[\prod_{v \in V(H_n(G))} \tau(\mathcal{Q}^{-1}(v))\Bigg],
\label{eqn:HTcount}
\end{align}
where $\tau$ computes the number of spanning trees on an input multi-graph, and $\mathcal{Q}^{-1}$ maps the subgraph of a quotient graph to the corresponding expanded (or pre-quotient) subgraph with respect to the hierarchy.
\proof{ 
To prove the equation for the above count, we first remark that for each tree, $t\in HT(G)$, we have that $\mathcal{Q}^n(t)$ is a spanning tree on $H_n(G)$ with, at most, a single edge connecting any two nodes, by construction. We remark that there is a one-to-one correspondence between each edge of a multi-graph $H_n(G)$ and the base graph $G$ and so we may associate each edge in the multi-graph with an edge on the base graph (see Figure~\ref{fig:coarseEdgeToFineEdge}). We begin by counting the number of spanning trees on the coarsest multi-graph, $H_\ell(G)$, according Kirchoff's theorem; this is the first term in the product of \eqref{eqn:HTcount}. 

We next consider the resulting multi-graph at the next coarsest layer. Since the edges of the coarsest spanning tree were associated with edges at the base level $E(G)$, we may also map the edge from $E(G)$ to the associated edge in the corresponding multi-graph of the next coarsest level.   
To continue our construction in the space of hierarchical trees, note that any edge we add to the next finer level, $\ell-1$, cannot span coarse level partitions because we would add a loop: either (i) there is already an edge across the partitions of $H_{\ell-1}(G)$ which will introduce a loop of two nodes into the quotient graph, $H_\ell(G)$, and thus $w_\ell(e) > 1$ and the graph would no longer be in $HT(G)$ or (ii) if we added an edge across another partition, we would be adding an edge in tree on the non-multi-graph tree $H_\ell(G)$, a tree must have a fixed number of edges.

Therefore, all edges we add when constructing a tree on $H_{\ell-1}(G)$ must be within a partition at level $\ell-1$. Within a partition, we must add enough edges to connect the space, but also generate no loops, meaning that we must construct a spanning tree within each partition.  The spanning trees will be connected to the rest of the graph due to the edges we have already added at the coarsest level that will span partitions at the finer level.  We repeat drawing nested trees within partitions down to the finest level.

We may again use Kirchoff's theorem to count the number of spanning trees within each multi-graph corresponding to a partition at all levels.  The choice of spanning trees within a partition is entirely independent across partitions and of the coarsest tree, hence we take a product over the number of spanning trees for each multi-graph partition at the next coarsest along with the original choice of tree at the coarsest level. The product of spanning trees at the next coarsest level corresponds to the inner product on the right hand side of \eqref{eqn:HTcount} when $n = \ell$.  Continuing recursively through the hierarchy, we arrive at the full expression of \eqref{eqn:HTcount}, where the last term ($n = 1$) corresponds to trees on the subgraphs induced by the finest partition made on the base graph $G$.\qed
}
\end{theorem}

\begin{figure}
\centering
\subcaptionbox{Coarse tree ($\mathcal{Q}(T_i)$)\label{sfig:coarseTree}}{\includegraphics[width=0.45\linewidth, clip = true, trim = {28cm 12cm 0cm 4cm}]{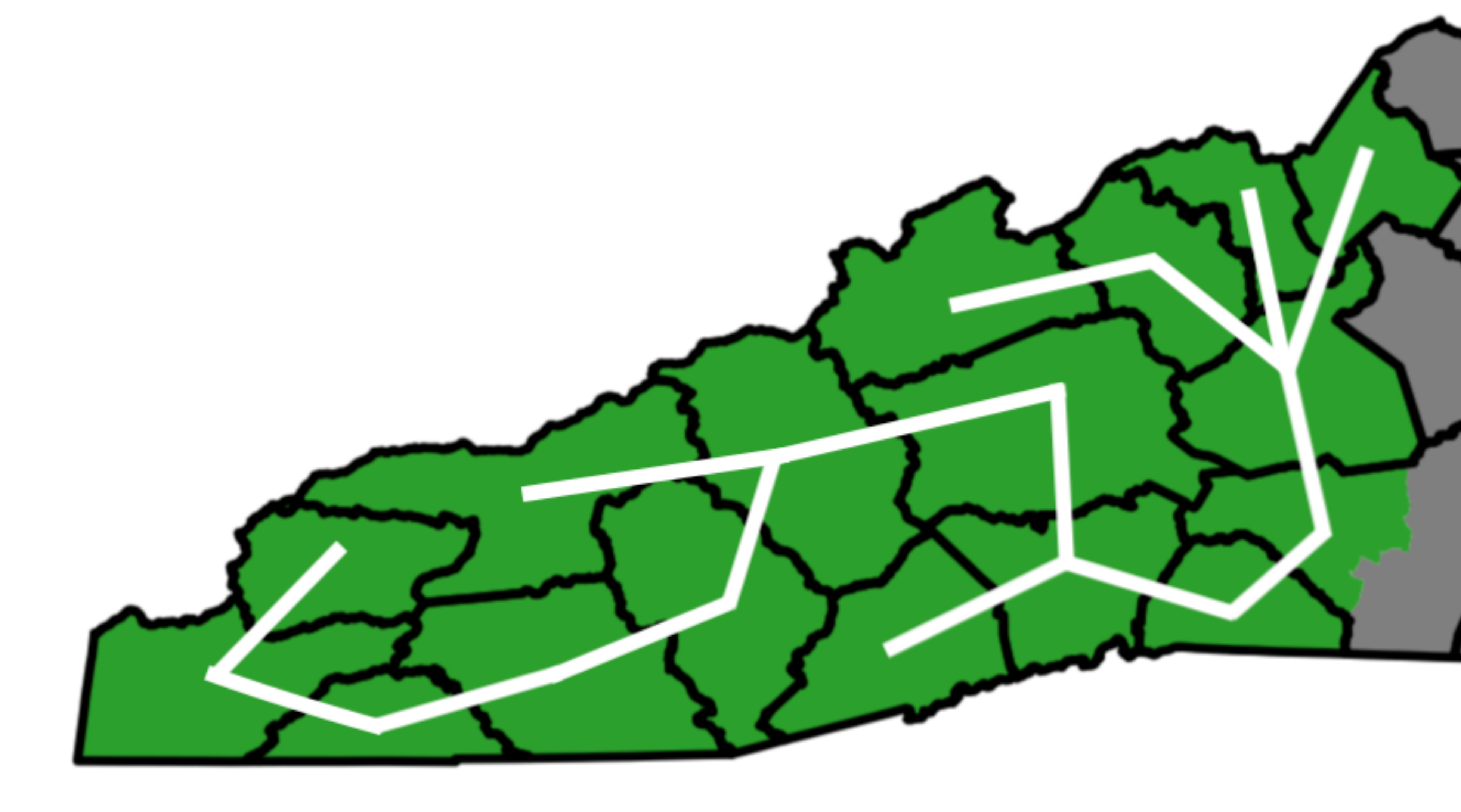}}\qquad
\subcaptionbox{Coarse edges at finer scale with tree on split node\label{sfig:expandedEdges}}{\includegraphics[width=0.45\linewidth, clip = true, trim = {28cm 12cm 0cm 4cm}]{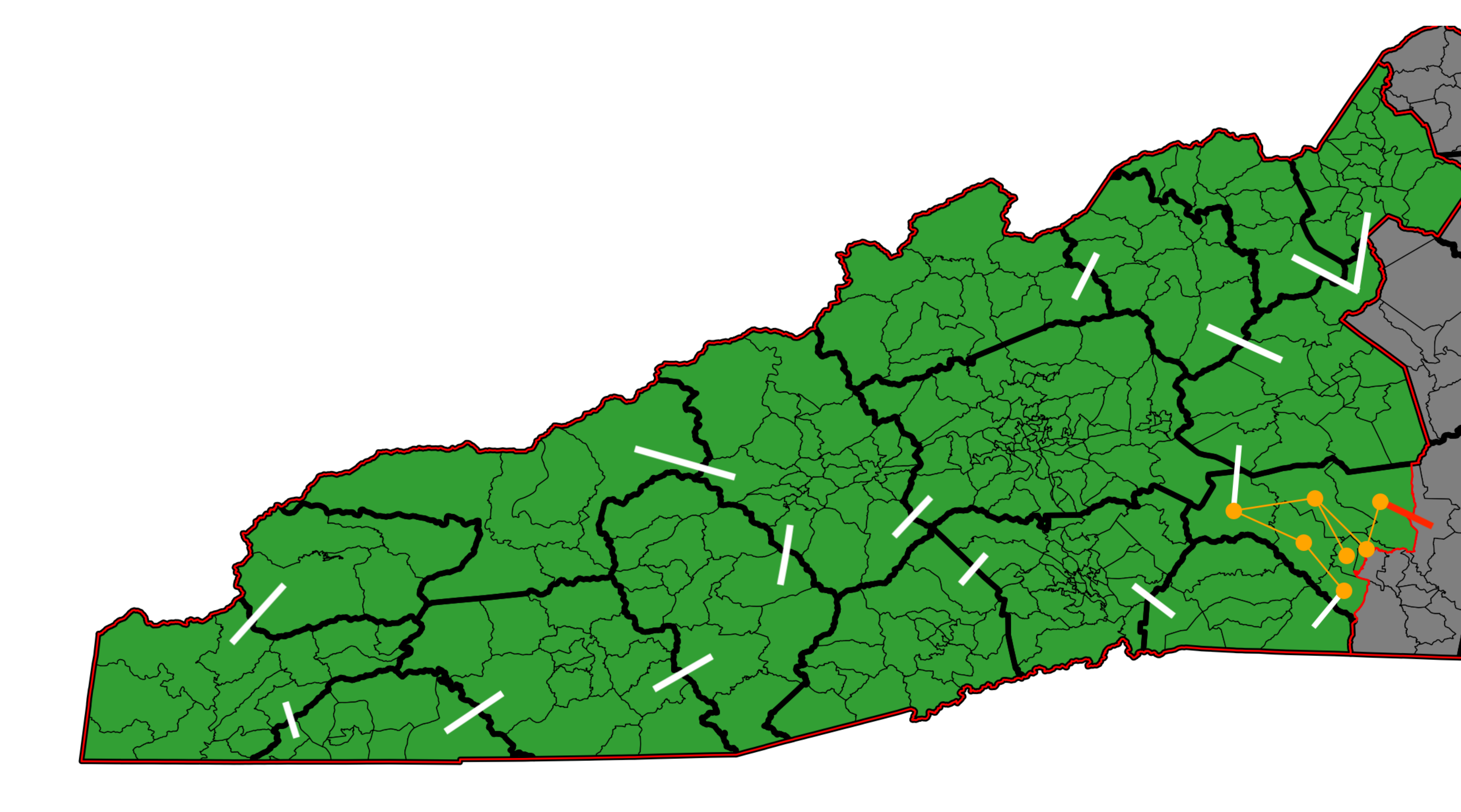}}\qquad
\caption{We display a coarse representation of a hierarchical spanning tree, $T_i$ on a green district \ref{sfig:coarseTree}.  These edges at the coarse scale are on a multi-graph that correspond to one of the edges at the finest scale \ref{sfig:expandedEdges}.  At the finer level, each expanded coarse node has an associated tree; an example is shown in the bottom right split node in orange. On un-split coarse nodes, it is often sufficient to know that there is \emph{some} tree at the finer level rather than knowing the particular tree; similarly, it can suffice to know that there is \emph{some} finer scale edge that links two coarse partitions without knowing the particular edge. The red edge shows the linking edge that spans between the green and gray districts.}
\label{fig:coarseEdgeToFineEdge}
\end{figure}

 \subsection{The structure of the measure}  \label{Sec:StructreuMeasure}
We collect a number of observations about the structure of the measure $P$ and various limiting cases in $\gamma$ and $\beta$.
We will write $P(\allT, L; \beta = b, \gamma = g)$ for the probability of seeing the districting $\allT$ in the distribution in \eqref{eq:Pdef} when $\beta=b$ and $\gamma=g$.

\subsubsection*{Uniform Measure on Hierarchical Forests.} When  $\gamma = 0$ and $\beta \rightarrow 0$, $P(\allT, L)$ converges to the uniform measure on the hierarchical forest of $d$ hierarchical trees which satisfy the constraints described by the score function $J$; that is to say $J(\xi) < \infty$. 

Keeping $\gamma=0$, if we were to set $\beta=0$ and use the convention that $0\times\infty=0$ in the exponent, the measure becomes
\begin{align}
P(\allT; \beta = 0, \gamma = 0) \propto 1,
\end{align}
which is to say we recover the uniform measure on the hierarchical forests, subject to no constraints.  

\subsubsection*{Uniform on All Graph Partitions} When $\gamma=1$ the distribution on graph partitions depends only on the factor involving $J$; that is, the probability of finding districting $\xi$ no longer depends on $\tau_\mathcal{H}(\xi)$. To see this, note that
\begin{align}
P(\xi) &\propto \sum_{L \in \mathcal{E}_L(\xi)} \sum_{\allT \in HT(\xi)} P(\allT, L) = e^{-\beta J(\xi)} \frac{\tree_{\mathcal{H}}(\xi) \mathcal{L}(\xi)}{\tree(\xi)_{\mathcal{H}}^\gamma \mathcal{L}(\xi)^\gamma},
\label{eqn:propxi}
\end{align}
where 
\begin{align}
HT(\xi) = HT(\xi_1)\times\cdots\times HT(\xi_n),
\end{align}
is the Cartesian product of all hierarchical trees, $HT(\xi_i)$, of subgraph $\xi_i$, and $\mathcal{E}_L$ is the set of allowable linking edge sets that correspond to the districting plan $\xi$ (note $\mathcal{L}(\xi)\equiv|\mathcal{E}_L(\xi)|$).

When $\gamma = 1$ and $\beta\rightarrow 0$, the measure becomes uniform on graph partitions subject to the absolute constraints given by the score function $J$. When $\gamma = 1$ and $\beta=0$ (as before, using the convention that $0\times\infty=0$), the measure is uniform on all graph partitions.

\subsubsection*{Intermediate Values of $\gamma$.} As can be seen in equation~\eqref{eqn:propxi}, as $\gamma$ becomes smaller, we favor partitions that have a larger product of tree counts on the districts.
In particular, when $\gamma = 0$ the chance of finding a districting plan with districts specified by $\xi$ is proportional to the product of the number of hierarchical trees on the subgraphs in $\xi$.  

The number of trees in any given partition may vary greatly.  To see this we note that the number of spanning trees grows faster than exponentially with the number of vertices in the graph, assuming the graph has average degree larger than 2 \cite{greenhillAverageNumberSpanning2017}, and that this provides an upper bound on the number of possible hierarchical trees. If the growth of the number of hierarchical trees relates to the rapid growth in the number of spanning trees, it may cause large disparates between the relative probabilities of different districting plans, as this ratio will be proportional to the product of hierarchical tree ratios
\begin{align}
\frac{P(\allD)}{P(\allD')} \propto \frac{\tree_{\mathcal{H}}(\allD)\mathcal{L}(\allD)}{\tree_{\mathcal{H}}(\allD') \mathcal{L}(\allD')} \frac{\tree_{\mathcal{H}}(\allD')^\gamma \mathcal{L}(\allD')^\gamma}{\tree_{\mathcal{H}}(\allD)^\gamma \mathcal{L}(\allD')^\gamma} e^{-\beta[ J(\xi)- J(\xi')]}.
\label{eqn:treeratio}
\end{align}
When taking a random walk through the state space using the Metropolis-Hastings algorithm, proposed states will usually have either far fewer or far more trees than the prior state, and the acceptance probability will be dominated by this ratio. This issue may be alleviated by modifying $\gamma$ in the interval $[0,1]$. There is a trade off; choosing $\gamma$ close to 0 leads to more similar probabilities and thus potentially better movement around the state space, but choosing $\gamma$ close to 1 leads to a distribution which is closer to uniform on the graph partitions rather than hierarchical forests.

\subsubsection*{Induced Measure on Partitions} We are primarily interested in the measure on partition $\xi$ of the graph as this maps to the redistricting application. However, for reasons we will make clearer in the discussion of the sampling algorithm, we have chosen to work on the extended state space of hierarchical forests.
It is instructive to pause and consider the relative structure of the measures on hierarchical forests and partitions. The following lemma shows that all forests which correspond to a given partition are equally likely to be sampled. In other words, the measure conditioned on a given partition is uniform on the hierarchical trees which correspond to that partition.

\begin{lemma}
  If two hierarchical forests $T$ and $T'$ and linking edge sets $L$ and $L'$ represent the same partition then their corresponding states have equal probability under the measure $P$. Given our notation, we may write that if $\xi(\allT)=\xi(\allT')$, then
  $P(\allT, L)=P(\allT', L')$.
\end{lemma}
\begin{proof}
  This follows from the fact that for a given hierarchical forest $T$, the score function $J$, the number of hierarchical trees $\tau_\mathcal{H}$, and the number of allowable linking edge sets $\mathcal{L}$ only depend on the partition $\xi(T)$. Since these are the only occurrences of $T$ in the definition of $P$, the result follows.
\end{proof}

\subsection{A multi-scale approach the the hierarchical problem}

We have primarily discussed the hierarchical structure while mentioning the multi-scale structure the hierarchy affords. In this section, we emphasize and expound on the multi-scale approach.  In principle, the state space has been defined by a forest of hierarchical trees $T$ defined at the finest scale.  However, we do not need to specify certain trees before they are needed. It suffices to know that there is \emph{some} tree defined on certain fine scale partitions without knowing \emph{which} tree it is. 

As an example, consider Figure~\ref{sfig:newForest}, in which we specified trees on two of the high level county nodes that could have been cut, but we did not specify the new trees in any of the other counties. Conceptually, such a tree is specified, but any specification would lead to the same probability distributions and edge choices.  Furthermore, as seen above in Theorem~\ref{thrm:HTcount}, the specific tree within a coarse node is independently chosen from the coarse tree and other trees at the same level (see also Theorem~\ref{thm:HTdraw} below); we may therefore avoid sampling it initially.  If, later in the computation, we do need to specify this tree, we may uniformly sample the multi-graph within the node.  

Because we avoid specifying the full details of the hierarchical trees, our method can be said to span multiple scales, rather than simply being hierarchical from the atomic units.  Throughout the rest of the paper, we will use the concept of a particular hierarchical tree with that of a partially specified hierarchical tree in which some of the nodes have not yet been assigned a state.

This choices allows both for faster sampling of uniform hierarchical trees and faster computation of hierarchical tree counts when compared with drawing uniform spanning trees or counting spanning trees on the base graph.

\subsubsection{Estimating computational complexity}
\label{sssec:complexHT}
The two most expensive steps in the algorithm are computing the number of hierarchical trees \eqref{eqn:HTcount} and in drawing a new uniform tree on a graph (see \ref{thm:HTdraw} and \ref{fig:overviewProposal}).  In the non-hierarchical framework with a graph with $N$ nodes, the two steps have a polynomial complexity of $O(N^{2.373})$ and $O(N^2)$, respectively.

We now estimate the computational complexity of computing the number of hierarchical trees in an idealized setting, assuming that each partition at every level in the graph of interest has $m$ vertices and $N = m^{\ell+1}$. Counting the number of trees on a graph with $m$ nodes using Kirchoff's theorem requires taking the determinant of an $(m-1)\times(m-1)$ matrix which has a complexity of $O(m^{2.373})$ \cite{aho1974design}. The formula in \eqref{eqn:HTcount} would require counting $(\sum_{i=0}^{\ell} m^i) = O(m^\ell)$ spanning trees, each of size $m$, leading to a computational complexity of $O(m^{2.373 \ell}) = O(N^{2.373}$; this is the same complexity associated with computing the number of spanning trees on the entire graph. 

In the multi-scale setting, however, we can significantly improve on this.  First, we precompute the number of trees for each partition, but do so only once (with the standard $O(N^{2.373})$ complexity).  When we draw a new forest (as in \ref{fig:overviewProposal}), we will, at most, split one new node per level per district.  The spanning tree counts of the remaining nodes on each district will have already been computed and stored.  This means that we must only compute two spanning tree counts per split node per level, which leads to a complexity of $O(\ell m^{2.373})$.

The relative complexities of computing the number of hierarchical trees depend on $m$ and $\ell$. However, in a multi-level framework on an arbitrary hierarchy, it may be reasonable to fix $m$ and let $\ell$ vary to change the overall size of the graph, $N$.  In this case we find that the complexity of the hierarchical tree count grows like $O(\log(N))$ which is significantly faster than the overall tree count complexity of $O(N^{2.373})$.

In terms of the complexity of drawing a new hierarchical tree we show that uniformly sampling hierarchical trees may be done by drawing a spanning tree at the coarsest level, $\mathcal{Q}^{\ell}(G)$ and then drawing a spanning tree within each partition from levels $\ell-1$ to 1  (see \ref{thm:HTdraw}).  We will only need to draw a tree on $O(1)$ nodes at each level when determining where to split the graph, meaning that this algorithm will have a complexity of $O(\ell m^{2})$ in the multi-scale framework.  Again, using the assumption that we can fix $m$ and vary $\ell$, the complexity of this algorithm will scale as $O(\log(N))$.
  
\section{Sampling From The Measure \texorpdfstring{$P$}{P}}
\label{sec:samplingFromP}

As already discussed, we will use a global Multi-Scale Merge-Split algorithm to propose moves to the standard Metropolis-Hastings algorithm. Our proposal is not itself reversible, but the resulting Markov chain given by the Metropolis-Hastings algorithm will be. Although one can always in theory use Metropolis-Hastings to create a reversible chain from any proposal method, it will fail to do so in practice if the rejection probabilities are too large or if calculating the necessary transition probabilities is computationally infeasible.

In the next section, we review the Metropolis-Hastings algorithm in this setting. In Section~\ref{sec:MS}, we give a full description of our Multi-Scale Merge-Split algorithm and many of the implementation details. We also further explain what is gained computationally by working on the space of hierarchical forests resolved over multiple scales as well as how employing the linking edge set expansion eases computational burdens.

\subsection{Metropolis-Hastings algorithm}

To sample from the measure $P$ on spanning forests and linking edges (defined previously in \eqref{eq:Pdef}), we use the Metropolis-Hastings algorithm with our Multi-Scale Merge-Split algorithm as the proposal method.
We will denote by $Q(S,S')$ the probability of starting from $S=(T,L)$, the multi-scale hierarchical forest $T$ and linking edge set $L$, and proposing the hierarchical forest $T'$ and edge set $L'$ with $S'=(T',L')$ using the Multi-Scale Merge-Split algorithm. In other words, if the current state of the chain is the hierarchical forest $T$ with linking edge set $L$, and $S=(T,L)$, then the measure $Q(S,\;\cdot\;)$ is the distribution of the next proposed move of the chain. Following the Metropolis-Hastings prescription, this move is accepted with a probability $A(S,S')$ defined by 
\begin{align}\label  {eq:A}
A(S, S') = \min\left(1, \frac{P(S')}{P(S)}\frac{Q(S', S)}{Q(S, S')}\right)=\min\left(1, e^{-\beta[ J(\xi')- J(\xi)]} \left[\frac{\tree_\mathcal{H}(\xi) \mathcal{L}(\xi)}{\tree_\mathcal{H}(\xi')\mathcal{L}(\xi')}\right]^\gamma  \frac{Q(S', S)}{Q(S, S')}\right),
\end{align}
and rejected with probability $1-A(S,S')$, where we have used the shorthand notation $\xi = \xi(T)$ and $\xi' = \xi(T')$. If the step is accepted, the next state is the proposed state; if the step is rejected, the next state does not change.
We conjecture that all states with non-zero probability are connected through a sequence of proposals; if this is true, this process with converge to sampling from the measure $P$ if run for sufficiently many steps.

\subsection{The Multi-Scale Merge-Split Algorithm}\label{sec:MS}
We now describe the Multi-Scale Merge-Split Markov Chain $Q$ introduced in previous section. As already mentioned, this Multi-Scale Merge-Split algorithm is specifically designed to have both forward and backward transition probabilities which can be efficiently computed.  From \eqref{eq:A}, we see that this is critical if it is to be used in the Metropolis-Hastings algorithm as a proposal.

\subsection*{The Multi-Scale Merge-Split proposal and probability \texorpdfstring{$Q$}{Q}} \label{sec:mergeSplitQ}
We first outline the Multi-Scale Merge-Split algorithm. We assume that the current state of the chain is the hierarchical forest $T$.  Our goal is to produce $T'$ which corresponds to merging two hierarchical trees in $T$ linked by an edge in $L$, then redivide the merged tree into two new hierarchical trees, along with a new linking edge, which satisfy constraints given by $J$.  We then calculate $Q(S,S')$ and $Q(S',S)$ in order to compute the acceptance ratio. Before diving into the details, we provide a high-level outline of this procedure.

Given spanning forest $T=(T_1,\cdots,T_n)$ and linking edge set $L$, we
\begin{enumerate}
    \item \textbf{Choose two trees to merge.} To do this choose a linking edge, $e_l\in L$.  Each vertex of the edge will correspond to a tree, $T_i$ and $T_j$ from $T$ which correspond to adjacent districts.
    \item \textbf{Draw a coarse spanning tree covering merged districts}. Draw a coarse-scale representation of a new hierarchical tree $\mathcal{Q}^\ell(T_{ij}')$ uniformly at random on coarsest level of the hierarchy induced on the merged districts, $\mathcal{Q}^\ell(\xi_{i,j})$;
    here, $\xi_{i,j}$ is the induced subgraph of the union of vertices in $T_i$ and $T_j$ on the finest scale. The induced graph $\xi_{ij}= (V_{ij}, E_{ij})$ is defined as $V_{ij}(\xi) = \{v \in V_0 \mid \xi(v) \in  \{i,j\} \}$ and $E_{ij}(\xi) = \{(v,u) \in E_0 \mid \xi(v), \xi(u) \in \{i,j\} \}$ (e.g. Figure~\ref{sfig:topLevelTree}).  Recall that $\mathcal{Q}^\ell(\xi_{i,j})$ maps the subgraph $\xi_{i,j}$ at the finest scale to the coarsest scale.
    \item \textbf{Determine removable edges and vertices to refine.} We will now split the tree into two trees, by either removing an edge or splitting a node. To this end, we identify edges of the coarse-scale representation of $T_{ij}'$ (denoted $\mathcal{Q}^\ell(T_{ij}')$), such that, once the edge is removed the two remaining trees each comply with some constraints, such as population balance.
    Similarly, we identify nodes such that the population of the node could be divided and joined with some edges of the tree to comply with the same constraints.
    Recall that since we are at the coarsest scale, it is more likely that we will find nodes to split than edges (e.g. see red nodes and edges in Figure~\ref{sfig:topLevelTree}). 
    \item \textbf{Refine specified vertices.} For the nodes that could be split, expand the nodes on the next finest level, $\ell-1$;  this is to say that we resolve the tree within the partitions one level down that corresponding to the cuttable nodes. For each node, draw a uniformly random spanning tree on the resulting inner-node multi-graph at the $\ell-1$ level (e.g. Figure~\ref{sfig:expandedTree}).
    \item \textbf{Continue to identify and refine until finest scale is reached.} Continue recursively by identifying edges and nodes on the newly expanded nodes. Continue until the base level is reached. At the base level we will only search for edges to remove, as nodes at the base level cannot be cut.
    Recall that each edge at the various levels of the hierarchy directly corresponds to an edge at the base level. At this point we have a collection of edges that may be removed and a partially specified hierarchical tree $T_{i,j}'$; it is partially specified at different scales because each of its coarsest nodes have only been resolved down to a certain scale (e.g. Figure~\ref{sfig:expandedTree}). 
    \item \textbf{Choose an edge to remove from those identified across all scales.} Removing any of the collected edges will result in two trees that comply with the specified constraints (such as population deviation).  Select one such edge, $e_l'$ and remove it from the new hierarchical tree $T_{ij}'$, leaving two new trees $T_i'$ and $T_j'$.  Place $e_l'$ in the linking edge set and remove $e_l$ from the linking edge set (e.g. Figure~\ref{sfig:newForest}).  Depending on the rule for drawing linking edges, re-sample linking edges between altered and non-altered districts.
    \item \textbf{Calculate forward and backward probabilities for acceptance ratio.} Calculate the probability of proposing $T_i'$, $T_j'$, and $e_l'$ starting from $T_i$, $T_j$, and $e_l$, $Q(S,S')$, and the reverse probability $Q(S',S)$.
\end{enumerate}
    
We now give more details about how each of these steps might be implemented.  The first step may be implemented in a variety of ways; for example, we may chose uniformly, or weight the choice by some property of the shared boundary between districts such as length or the values of the score function $J$. Depending on the constraints, we may only wish to choose from a subset of the linking edges.

The second step is achieved by recursively employing Wilson's algorithm from the coarsest hierarchical description to the finest, which employs loop-erased random walks.  For vertices that cannot be split, we do not need to specify which tree is drawn within the vertex as it suffices to know that \emph{some} tree was drawn. Thus we can store a coarser multi-scale description of the tree.  We prove that the above algorithm uniformly samples hierarchical tree space below in Theorem~\ref{thm:HTdraw}.

For the third and fourth steps, we focus on population constraints and coarse node constraints (e.g. county constraints) in the current work. These steps involve recursively employing a simple depth-first search along the current multi-scale tree representation with exit criteria based on the remaining population within a search branch; we look both for edges to split and vertices with population that could be divided to satisfy population constraints.  If splitting the vertex or cutting the edge would violate county split provisions we will omit it. Vertices that could be cut are specified with a recursive Wilson's algorithm which allows the depth-first search to continue at the finer level. In the redistricting problem we may also take advantage of the added geographic information embedded in the original map to quickly determine which nodes are possible to cut (see Section~\ref{ssec:expandNode}).

Choosing the specific edge to cut in the fifth step may be done uniformly or with a weighted distribution that might, for example, favor more equal populations, perhaps tempered according to $J$. 

In many ways, step six is the most involved. It also critically depends on the details of how the previous four steps were implemented. It is in this step that we see why the choice of expanding to the space of hierarchal forests, rather than partitions, is important, as well as how keeping linking edges further simplifies computation.

\subsection*{Calculating \texorpdfstring{$Q(S,S')$}{Q(S,S')} and \texorpdfstring{$Q(S',S)$}{Q(S',S)}}
Given the state $S=(T,L)$, let us denote by $p(e_l \mid S)$ the probability 
from step one of picking $e_l\in L$ and the associated pair of adjacent spanning trees $T_i$ and $T_j$ to merge. This probability is simple to calculate for most reasonable choices of how to perform step 1 and will be based on the linking edges $L$, and possibly the spanning forest $T$. We will let $T'$ (in $S'$) denote $T$ (in $S$) with the $T_i$ and $T_j$ replaced by $T_i'$ and $T_j'$ respectively. Then
\begin{align}\label{eq:MergeSplitQ}
Q(S, S') = p(\{i,j\} \mid S) q(\{T_i, T_j\, e_l\}, \{T_i', T_j', e_l'\}).
\end{align}
where $q(\{T_i, T_j, e_l\}, \{T_i', T_j', e_l'\})$ is the chance that the merging of $\{T_i, T_j, e_l\}$ and subsequent splitting produces the replacement spanning trees and edge $\{T_i', T_j', e_l'\}$.

Recall from above that the new linking edge was once part of the merged tree. Thus, to compute $q(\{T_i, T_j, e_l\}, \{T_i', T_j', e_l'\})$, we must examine (i) the chance of drawing the tree defined by edges $E(T_i')\cup E(T_j') \cup \{e_l'\}$ on $\xi_{ij}$, which we will denote as $T_{(T_i', T_j', e_l')}$; and (ii) the chance of choosing $e_l'$ as the new linking edge from the tree $T_{(T_i', T_j', e_l')}$.  We will show that the probability of drawing a single tree on $\xi_{ij}$ from $HT(\xi_{ij})$ is $\tau_{\mathcal{H}}(\xi_{ij})^{-1}$ below in Theorem~\ref{thm:HTdraw}. We will denote the second probability that we cut the tree $T_{(T_i', T_j', e_l')})$ at edge $e_l'$ be $P_{cut}(e_l' \mid T_{(T_i', T_j', e_l')})$. We will specify $P_{cut}$ below in Section~\ref{ssec:pcut}.

Putting this all together, we find that the probability of the proposing $T_i'$, $T_j'$ and $e_l'$ from $T_i$, $T_j$ and $e_l$ is
\begin{align}
q(\{T_i, T_j, e_l\}, \{T_i', T_j', e_l'\}) = \frac{1}{\tau_{\mathcal{H}}(\xi_{ij})} P_{cut}(e_l' \mid T_{(T_i', T_j', e_l')}).
\label{eqn:twoDistProposal}
\end{align}

\subsection*{Properties of the Merge-Split Proposal \texorpdfstring{$Q$}{Q}} 
First observe that when the Merge-Split proposal from \eqref{eq:MergeSplitQ} is inserted into the formula for the acceptance probability from \eqref{eq:A}, we obtain
 \begin{align}\label{AQ}
    A(S, S') = \min\left(1, e^{-\beta(J(\xi')-J(\xi))}\left[\frac{\tau_\mathcal{H}(\xi) \mathcal{L}(\xi)}{\tau_\mathcal{H}(\xi') \mathcal{L}(\xi')}\right]^\gamma\frac{p(e_l' \mid S')} {p(e_l \mid S)}\frac{q(\{T_i', T_j', e_l'\}, \{T_i, T_j, e_l\})}{q(\{T_i, T_j, e_l\}, \{T_i', T_j', e_l'\})}\right). 
 \end{align}
From this we see that if $\gamma=0$, one does not need to calculate the $\tau_\mathcal{H}(\xi)$ and $\tau_\mathcal{H}(\xi')$ factors, which reduces the computational costs. Furthermore, as mentioned before, the difference between $\tau_\mathcal{H}(\xi)$ and $\tau_\mathcal{H}(\xi')$ may be a principle reason for a lower acceptance rate when $\gamma$ is closer to 1.
 
Inserting \eqref{eqn:twoDistProposal} into \eqref{AQ},  we see that the ratios of the Merge-Split probabilities can be written as
 \begin{align}
   \label{eq:qq}
   \frac{q(\{T_i', T_j', e_l'\}, \{T_i, T_j, e_l\})}{q(\{T_i, T_j, e_l\}, \{T_i', T_j', e_l'\})}= \frac{P_{cut}(e_l \mid T_{(T_i, T_j, e_l)})}{P_{cut}(e_l' \mid T_{(T_i', T_j', e_l')})}\,.
 \end{align}
In particular, we see that the $\tau_\mathcal{H}(\xi_{ij})$ factors in each $q$ expression cancel and hence need not be computed. The acceptance ratio becomes 
\begin{align}\label{eq:Asimplified}
  A(S, S') = \min\left( 1, 
  e^{-\beta [ J(\xi') - J(\xi)] } \left[\frac{\tau_\mathcal{H}(\xi_i) \tau_\mathcal{H}(\xi_j) \mathcal{L}(\xi)}{\tau_\mathcal{H}(\xi_i') \tau_\mathcal{H}(\xi_j') \mathcal{L}(\xi')}\right]^\gamma \frac{p(e_l' \mid S')}{p(e_l\mid S)} \frac{P_{cut}(e_l \mid T_{(T_i, T_j, e_l)})}{P_{cut}(e_l' \mid T_{(T_i', T_j', e_l')})}
  \right)
\end{align}
where we have used the fact that the spanning forests $T$ and $T'$ only differ in the $i$th and $j$th trees so
\begin{align*}
  \frac{\tree_\mathcal{H}(\xi(T))}{\tree_\mathcal{H}(\xi(T'))}&= \frac{\tree_\mathcal{H}(\xi_i)  \tree_\mathcal{H}(\xi_j)}{\tree_\mathcal{H}(\xi_i')  \tree_\mathcal{H}(\xi_j')}.
\end{align*}

We may further simplify this equation by using \eqref{eqn:HTcount} and writing
\begin{align}
\label{eq:countRatioSimplified}
  \frac{\tree_\mathcal{H}(\xi(T))}{\tree_\mathcal{H}(\xi(T'))}
 = \frac{\tree(H_\ell(\xi_i)) \tree(H_\ell(\xi_j)) \prod_{n=1}^\ell \prod_{v\in V(H_n(\xi_i)) \cup V(H_n(\xi_j))} \tree(\mathcal{Q}^{-1}(v))} {\tree(H_\ell(\xi_i')) \tree(H_\ell(\xi_j')) \prod_{n=1}^\ell \prod_{v\in V(H_n(\xi_i')) \cup V(H_n(\xi_j'))} \tree(\mathcal{Q}^{-1}(v))}.
\end{align}
The majority of terms under the products will be shared across the numerator and denominator.  Indeed a Merge-Split operation will take one cascade of splits (e.g. split county and precinct) and transfer it to another cascade of splits (e.g. there will be one different in a split county, one different in a split precinct, etc...). The tree counts for the nodes that are not changed in the Merge-Split process will cancel. There will only be one node per level per district that will be altered in this procedure and for which we will need to recompute the tree count.  If we denote such nodes as $v_{i:ij}^n$ for the node at level $n$ in $\xi_i$ that is cut across $\xi_i$ and $\xi_j$, and $v_{ij}^n$ as the merged cut node in $H_n(\xi_{ij})$ (defined precisely as $\mathcal{Q}(\mathcal{Q}^{-1}(v_{i:ij}^n) \cup \mathcal{Q}^{-1}(v_{j:ij}^n))$), then \eqref{eq:countRatioSimplified} simplifies to 
\begin{align}
\label{eq:countRatioSimplifiedPair}
  \frac{\tree_\mathcal{H}(\xi(T))}{\tree_\mathcal{H}(\xi(T'))}
  = 
  \frac{
    \tree(H_\ell(\xi_i)) \tree(H_\ell(\xi_j))
    \prod_{n=1}^\ell \tree(\mathcal{Q}^{-1}(v_{i:ij}^n)) \tree(\mathcal{Q}^{-1}(v_{j:ij}^n)) \tree(\mathcal{Q}^{-1}({v_{ij}^n}'))
  } {
    \tree(H_\ell(\xi_i')) \tree(H_\ell(\xi_j'))
    \prod_{n=1}^\ell \tree(\mathcal{Q}^{-1}({v_{i:ij}^n}')) \tree(\mathcal{Q}^{-1}({v_{j:ij}^n}')) \tree(\mathcal{Q}^{-1}(v_{ij}^n))}.
\end{align}
We remark that ${v_{ij}^n}' \in V(H_n(\xi_i)) \cup V(H_n(\xi_j))$ when ${v_{ij}^n}'$ is not split in $T$ (and similarly for $v_{ij}^n$);
if the node is split in both $T$ and $T'$, then ${v_{ij}^n}' = {v_{ij}^n}$ implying $\tree(\mathcal{Q}^{-1}({v_{ij}^n}')) = \tree(\mathcal{Q}^{-1}({v_{ij}^n}))$ and the terms cancel in \eqref{eq:countRatioSimplifiedPair}.  These two facts demonstrate why there is equality between \eqref{eq:countRatioSimplified} and \eqref{eq:countRatioSimplifiedPair}. 

If there is no split node at a certain level, say $k$, there will be no split nodes between the districts at any level finer (i.e. less) than $k$. In the simplification of \eqref{eq:countRatioSimplifiedPair}, this means that $v_{i:ij}^n$, $v_{j:ij}^n$, and $v_{ij}^n$, would not be well defined when $n\leq k$. Informally, we would omit these terms from the equation. Formally, we set the tree counts to $\tree(\mathcal{Q}^{-1}(v_{ij}^k)) = \tree(\mathcal{Q}^{-1}(v_{i:ij}^k)) = \tree(\mathcal{Q}^{-1}(v_{j:ij}^k)) = 1$ in \eqref{eq:countRatioSimplifiedPair}, for $k\leq n$ or, equivalently, when the split node is undefined.

As mentioned previously, the ratio between spanning tree products, $\tree_\mathcal{H}$, may be large between districting plans. This disparity is eliminated when $\gamma = 0$, but setting $\gamma = 0$ favors sampling partitions with higher values of $\tree_\mathcal{H}$. The parameter $\gamma$ is presented as a smoothly varying parameter because it demonstrates how one may use a tempering (e.g. simulated or parallel tempering) scheme to vary $\gamma$ from 0 to 1 across multiple chains in an extended product measure. 
We have chosen a form of the measure which essentially depends only on the partition $\xi(T)$ induced by the spanning forest $T$ (noting that the number of linking edge sets is also determined by $\xi(T)$). 
Additionally, when $\gamma = 0$ there is no need to compute the number of trees on the new districts, as the products $\tree_\mathcal{H}$ do not explicitly appear in the measure nor the proposal ratio.  

\subsection*{Lifting from partitions to tree partitions with linking edges eases computational complexity} 
Note that \eqref{eqn:twoDistProposal} reduces simply to finding the number of hierarchical trees along with the probability of choosing a given cut.  As we will specify below with implementation details (see Section~\ref{sec:implement}), this is a fairly simple calculation.  We pause to compare this calculation with Recom \cite{deford2019recombination} and our previous Merge-Split algorithm \cite{carter2019merge} in order to demonstrate the computational advantage of including a linking edge set.

In Recom, proposals are made across partitions rather than across trees.  To reach a particular partition would require summing over all trees that could have lead to the same partition; in each element of the sum, the probability of choosing an edge could change, meaning that one would have to compute
\begin{align}
q(\{\xi_i, \xi_j\}, \{\xi_i', \xi_j'\}) = \sum_{T_i' \in ST(\xi_i')} \sum_{T_j' \in ST(\xi_j')} \sum_{e\in E(\xi_i', \xi_j')} \frac{1}{\tau_\mathcal{H}(\xi_{ij})} P_{cut}(e \mid T_{(T_i', T_j', e)}),
\end{align} 
in the hierarchical context,
where $E(\xi_i', \xi_j')$ is the set of edges in $E_0$ connecting $\xi_i'$ to $\xi_j'$.
This sum is computationally infeasible.  

The key insight of the Merge-Split was to expand the state space to the spanning forest which simplifies the above summation to
\begin{align}
\label{eqn:ogMSFramework}
q(\{T_i, T_j\}, \{T_i', T_j'\}) = \sum_{e\in E(\xi_i, \xi_j)} \frac{1}{\tau_\mathcal{H}(\xi_{ij})} P_{cut}(e \mid T_{(T_i', T_j', e)}),
\end{align}
which turns a summation with a number of terms into a sum along the boundary between partitions.

In the current work, we have expanded the state to further simplify the sum.  By keeping track of which edge was cut to split the trees, the proposal probability becomes the single term in \eqref{eqn:twoDistProposal}, 
\begin{align*}
q(\{T_i, T_j, e_l\}, \{T_i', T_j', e_l'\}) = \frac{1}{\tau_\mathcal{H}(\xi_{ij})} P_{cut}(e_l' \mid T_{(T_i', T_j', e_l')}),
\end{align*}
which further simplifies the computation at the expense of further expanding the state space.

\begin{remark}
The original Merge-Split framework shown in \eqref{eqn:ogMSFramework} without tracking a linking edge set is entirely compatible with the current multi-scale framework.  There is, however, the added burden of needing to expand more nodes as each of the different merged-tree might specify different nodes to expand.
\end{remark}

\section{Implementation details of Multi-Scale Merge-Split} 
\label{sec:implement}

Many of the above steps may be achieved with a variety of choices.  In this section, we specify the choice or choices we explore in the current work, as well as discuss some other possible methodologies. For example, there are many choices when picking which adjacent districts to merge (denoted $P(\{i,j\}|S)$ above) -- one may uniformly choose amongst adjacent (linked) districts or weight the choice by the shared border length, the shared number of conflicted edges, or some heuristic of the acceptance probability.  In the current work we make this choice by uniformly selecting one of the linking edges.

\subsection{Determining the ratio of tree counts in a partition}
If $\gamma \neq 0$ we must compute the ratio of the number of hierarchical trees on each subgraph induced by $\xi_i$. This is accomplished via Kirchoff's theorem and is, algorithmically, the slowest step of the method, however (i) we achieve the hierarchical speed-up mentioned in Section~\ref{sssec:complexHT} and (ii) we pre-compute the tree count on any preserved element of the hierarchy (such as number of precinct trees with in a county) and track the tree count on any split element of the hierarchy.  In terms of complexity, we will only have to recompute tree counts on the new split elements, of which there will only be one at each level (e.g. one new split county, one new split precinct, etc.), as well as counting the number of trees on the new coarse level structure. In general, this means that we need only compute one tree count per level per district (one for each district intersected with the split node at each level), and one tree counts per district on the coarsest scale of each graph. This still leads to an aggregate computational complexity of $O(\ell)\sim O(\log(N))$ (the same as computing the total number of hierarchical trees), however, the leading order constant will be significantly reduced due to the pre- and tracked computations on the remaining nodes.

To illustrate an example of the above complexity, we revisit Figure~\ref{sfig:newForest}.  In counting the number hierarchical trees associated with the new gray and green partitions, we will need to compute the number of ways the coarse trees could have been drawn (accounting for the split county nodes at the coarse level), as well as how many precinct trees could be drawn in each color of the split gray/green county (a total of four applications of Kirchoff's algorithm).  The remaining hierarchical tree counts, such as the number of precinct trees within each green county or the number of trees in the green district of the blue/green county, will have already been computed.  When computing the acceptance ratio \eqref{eq:Asimplified}, many of these computations will not even need to be examined as demonstrated in \eqref{eq:countRatioSimplifiedPair}. We further remark that the last terms in the products will either cancel or have already been stored in memory, helping to illustrate the above claim that we must compute (i) the new coarse level tree counts per district (the two denominator tree counts outside the product) and (ii) one tree count per district (the first two terms within the product in the denominator).

\subsection{Determining the probability of cutting an edge}
\label{ssec:pcut}
When choosing what edge to cut, we must specify the probability of cutting edge $e$ given merged tree $T_{ij}$, which we denote as $P_{cut}(e | T_{ij})$.  Perhaps the simplest implementation is to uniformly choose an edge from the set of edges such that the the cut leads to two trees with populations within the constraints set out by $J$, along with any other constraints we might impose. Let $E_c(T)$ denote the edges, such that if a single edge was removed from $T$, the remaining two trees would satisfy a set of constraints. Then \begin{align}
  P_{cut}(e | T) = 
  \begin{cases} 
    |E_c(T)|^{-1} & e \in E_c(T),\\
    0 & \text{ otherwise.}
  \end{cases}
\end{align}
We adopt this approach in the present work.  To determine $E_c(T)$, we begin by using the algorithm from \cite{carter2019merge} in which we randomly root the merged tree and then traverse from the farthest leaves in order to compute the population (and any other relevant conditions) when removing the edge.  In addition to searching for edges to cut, we also determine whether a node could be cut (see below in Section~\ref{ssec:expandNode}).  For each node that might be cut, we expand the node at the next level of the hierarchy.  The edges on the multi-graph connecting the coarse nodes correspond to edge connections between the finer scale nodes (since they are mapped to fine scale edges spanning partitions) and these specified edges are used to connect nodes or partitions representing the coarse graph to the finer subgraph on a partition.  At each finer subgraph, we draw a new tree (see Figure~\ref{sfig:expandedTree}); together with the edges across coarse nodes/partitions we arrive at a multi-level graph and repeat the above algorithm, searching for nodes and edges that can be cut.
We recursively expand nodes until we arrive at the base level and can only add edges.  The set of edges collected across all steps forms $E_c(T)$.

\subsection{Uniformly sampling/specifying a hierarchical tree}
Before illustrating how to check if a node should be expanded, we first turn to sampling the space of hierarchical trees.
As already mentioned, uniform hierarchical trees will be drawn using Wilson's on disjoint hierarchical subgraphs. There are several implementations of Wilson's algorithm and we use the same algorithm as in \cite{carter2019merge}, with the exception that we will employ the algorithm for multi-graphs.  

In steps three and four of the algorithm (see Section~\ref{sec:MS}) we claim that recursively employing Wilson's algorithm on successively finer features will sample from the space of hierarchical trees, however we have yet to specify what distribution this process will sample from.  We state and prove the following

\begin{theorem}
\label{thm:HTdraw}
Given a graph $G$ with hierarchy $\mathcal{H}(G)$, sampling independent spanning trees with Wilson's algorithm on $H_{\ell}$, and on $\mathcal{Q}^{-1}(v)$ for every $v\in V(H_n(G))$ over all $n \in \{1, ..., \ell\}$, will uniformly sample a tree from $HT(G)$.
\proof{
Consider any $t \in HT(G)$.  Consider $\mathcal{Q}^\ell(t)$ which is a spanning tree on the multi-graph $H_\ell(G)$ by construction.  
Since Wilson's algorithm uniformly samples spanning trees on multi-graphs, the probability of drawing $\mathcal{Q}^\ell(t)$ is $\tau(H_\ell(G))^{-1}$.  Similarly, consider any $n \in \{1, ..., \ell\}$, $v\in V(H_n(G))$, and the sub-multi-graph given by $\mathcal{Q}^{-1}(v)$.  Consider also the subgraph of $\mathcal{Q}^{n-1}(t)$ induced by $V(\mathcal{Q}^{-1}(v))$.  The induced subgraph is a tree, since it is part of the large tree $\mathcal{Q}^{n-1}(t)$, and employing Wilson's algorithm on the subgraph $\mathcal{Q}^{-1}(v)$ will sample the induced tree with a probability of $\tau(\mathcal{Q}^{-1}(v))^{-1}$.  

Each of the induced subgraphs are independent of one another and of $\mathcal{Q}^\ell(t)$, which is to say that the edges in one graph do not correspond to the edges of another when mapped back to $t$ through the multi-graph edge correspondences with the base graph.  Therefore, the probability of drawing $t$ under this procedure is 
\begin{align}
\left(\tau(H_\ell(G)) \prod_{n = 1}^\ell \Bigg[\prod_{v \in V(H_n(G))} \tau(\mathcal{Q}^{-1}(v))\Bigg]\right)^{-1} = \left(\tau_{\mathcal{H}}(G)\right)^{-1},
\end{align}
by equation \eqref{eqn:HTcount}.  In short, the tree will have been uniformly sampled from the space of hierarchical trees by applying Wilson's algorithm to the coarsest quotient graph and all sub-multi-graphs induced by partitions across each level of the hierarchy.
\qed}
\end{theorem}

\subsection{Determining which nodes to expand}
\label{ssec:expandNode}
Upon drawing a tree at a given level, we must determine which nodes to expand at the finer scale. The minimal considerations on whether or not we should expand a node is to check whether it is possible that (i) a subset of the tree connections to the node, along with some fraction of the node's population, would be within the acceptable population limits; and (ii) the complementary population satisfies the same limit.  This is to say that given some multi-level tree, $t_m$, and node, $v\in V(t_m)$, we can consider the remaining connected components of $t_m$ upon removing $v$. We denote the set of these components as $\mathcal{C}$ and note that for every $C\in\mathcal{C}$ there exists a unique neighbor of $v$, $v_n \in C$, i.e. $(v_n, v)\in E(t_m)$. A necessary condition for the node to be expanded is that there is some subset $\mathcal{C}_i \subset \mathcal{C}$ such that these components together with some fraction of the node will be within the population constraints, as will the complementary components.  Letting $\pop_I-\epsilon$ be the smallest allowable population and $\pop_I + \epsilon$ be the largest allowable population, we obtain the condition that there must exist some $\mathcal{C}_i$ such that
\begin{subequations}
\label{eq:expandNecCond}
\begin{align}
\pop_I - \epsilon -  \pop(v) \leq& \sum_{u \in \mathcal{C}_i} \pop(u) < \pop_I + \epsilon, \text{ and }\\
\pop_I - \epsilon -  \pop(v) \leq& \sum_{u \in t_m\backslash (\mathcal{C}_i \cup \{v\})} \pop(u) < \pop_I + \epsilon.
\end{align}
\end{subequations}

One possible solution to determining which nodes to expand would simply be to iterate through all possible combinations of $\mathcal{C}$ and expand the node if we find any possible subset that potentially achieves proper population balance.  We can, however, take advantage of two observations in order to (i) only check certain nodes, and (ii) when checking a node, exam only a subset of $\mathcal{C}$ by taking advantage of the geographic structure that is used to define the graph hierarchy.

To the first point, we check if any one of the components $C\in\mathcal{C}$ has $\pop(C) > \pop_I + \epsilon$; if it does, conditions \eqref{eq:expandNecCond} cannot be satisfied.  To the second point, we store the clockwise neighbors associated with the adjacency criteria and the original districting graph at each level of the heirarchy (see Figure~\ref{sfig:mockPrecinctGeoGraph}).  We take this list and omit any edge not found in the multi-level tree $t_m$; we remark that for nodes that appears multiple times, the edge at the finest scale will be specified in the multi-graph, and we will therefore can determine which of the multiple segments we should keep.  For example, in Figure~\ref{sfig:mockPrecinctGeoTree}, we would obtain the oriented subset of neighbors to be $\{3,1,5,6,7\}$.

We then set binary lists to corresponds district assignments of each neighbor connected through $t_m$: For example, `ABBBB' corresponds to $\mathcal{C}_i = \{3\}$ in Figure~\ref{sfig:mockPrecinctGeoGraph}, and `BAABB' corresponds to $\mathcal{C}_i = \{\{1,2\}, \{4,5\}\}$.  
We remark that if a district is to be connected, we cannot allow for overlapping paths that would break a connection.  For example, in Figure~\ref{sfig:mockPrecinctGeoTree}, the assignment `ABABB' would force us to find a connected path between regions 1 and 6, and 5 and 7, which would be impossible.

Therefore, we search the possible choices of $\mathcal{C}_i$ by first checking if $\mathcal{C}_i = \emptyset$ meets criteria \eqref{eq:expandNecCond}, next we begin to iterate through a queue in a breadth first search with all single `A' assignments (from the example `ABBBB', `BABBB', `BBABB', `BBBAB' and `BBBBA').
For each element in the list, we check criteria \eqref{eq:expandNecCond}.  If the criteria is met we must expand the node; if $\pop(\mathcal{C}_i) < \pop_I + \epsilon$, we expand the `A' assignment region to the right (on the torus) and add it to the queue e.g. `ABBBB'$\rightarrow$`AABBB' and `BBBBA' $\rightarrow$ `ABBBA'.  

We do not add anything to the queue if the number of A's exceeds half the length of the nodes since, by symmetry, such an assignment will have already been explored. To see this through an example, note that `AAABB' is equivalent to `BBBAA' which will be added to the queue from `BBBAB.'

\begin{figure}
\centering
\subcaptionbox{Neighboring regions\label{sfig:mockPrecinctGeoGraph}}{\includegraphics[width=0.3\linewidth, clip = true, trim = {0cm 0cm 0cm 0cm}]{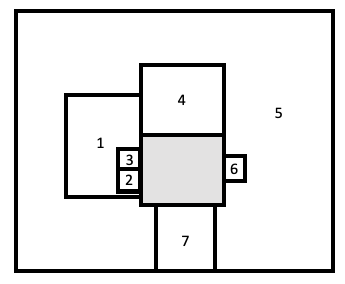}}\qquad
\subcaptionbox{With multi-level tree\label{sfig:mockPrecinctGeoTree}}{\includegraphics[width=0.3\linewidth, clip = true, trim = {0cm 0cm 0cm 0cm}]{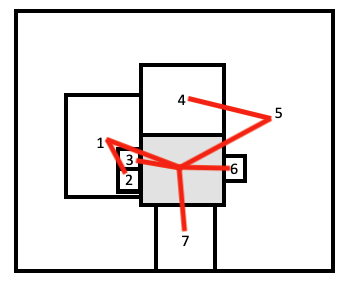}}\qquad
\caption{We display an example of mock neighboring geographical regions at a given level (A). The central grey district stores its neighbors in a clockwise ordered list such as \{1,2,3,1,4,5,6,5,7,5\}.  For a given tree at this level (B), we can obtain an oriented neighbor list such as \{3,1,5,6,7\}.}
\end{figure} 

\subsection{Defining linking edge sets (in concert with plan constraints)}
\label{ssec:linkingEdgeSets}
To close this section, we discuss several methods of defining and counting linking edge sets. In general, there are many possible choices in sampling and constraining linking edges. We do, however, require several properties for all linking edge sets, which are

\begin{itemize}
  \item Across any two districts, there can be, at most, one linking edge.  
  \item A linking edge must connect two districts across a hierarchical node that shares the districts.
\end{itemize}

The first condition ensures that two districts together with all linking edges form a hierarchical tree. This ensures the merged hierarchical tree has a non-zero probability of being proposed in the reverse proposal probability of the Metropolis-Hastings algorithm. The second condition ensures that districts that share a node within the hierarchy may be merged so that the split hierarchy may become whole within a proposal.

Perhaps the simplest method for drawing a linking edge set is to place a single linking edge across all adjacent districts while conforming to the requirements listed above.  Such a choice means that two adjacent districts \emph{must} only share one node at each level in the hierarchy, which is something that is seen in practice, but is not a necessary condition. We describe how one can remain in the framework by introducing a dynamic hierarchy in the Appedix~\ref{apdx:DynamicHierarchies}, but do not computationally explore this in the current work.

In this set up, and under the Multi-Scale Merge-Split algorighm, we erase all linking edges that link altered districts to non-altered districts, and then uniformly resample new linking edges to the newly formed districts (see Figure~\ref{fig:newLinkingEdgeSet}).

\begin{figure}
\centering
\includegraphics[width=0.85\linewidth, clip = true, trim = {0cm 0cm 0cm 0cm}]{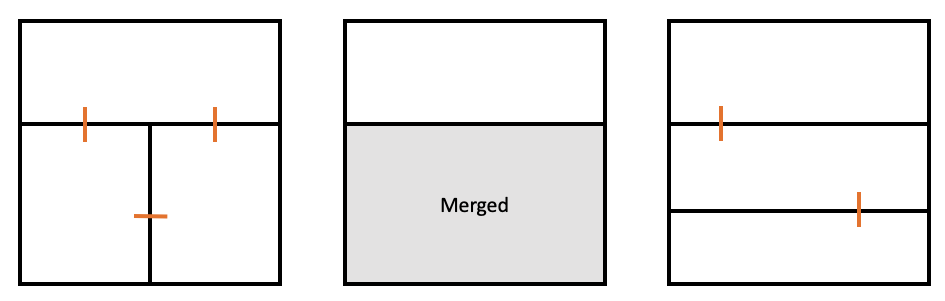}
\caption{We display an example of three districts with linking edges across all adjacent districts (left). The two bottom districts are merged and both the merged linking edges and the connecting linking edges are erased (middle).  Upon drawing a new district a new linking edge is chosen uniformly across the new top two districts (right).}
\label{fig:newLinkingEdgeSet}
\end{figure}

When $\gamma > 0$ in the measure, we must ensure that it is possible to compute the number of possible linking edge sets that correspond to the partition. In this algorithm, we simply take the product of the number of linking edges that could be drawn across adjacent districts.  This number will depend on whether or not there is a node in the graph hierarchy that is shared across the districts, and, if there is, the minimum level in the hierarchy that contains a split node.  Formally, we define
\begin{align}
\underline{n}_{ij} = \min\big(\{n+1 | \exists (u_i,u_j)\in E(H_n) \: s.t. \: &\mathcal{Q}(u_i) = \mathcal{Q}(u_i),\\
&\forall v\in V(\mathcal{Q}^{-n}(u_i)), v\in \xi_i, \text{ and } \\
&\forall v\in V(\mathcal{Q}^{-n}(u_j)), v\in \xi_j \} \cup 
\{\ell+1\}\big), \nonumber
\end{align}
as the finest level of the hierarchy that is split across districts $i$ and $j$.  The reason for taking the union with the single element set $\{\ell+1\}$ is because when there is no split element of any level of the hierarchy across the adjacent districts, the first set will be empty; furthermore, the first set of the right hand side has elements that are bounded above by $\ell$.  Next, if $\underline{n}_{ij}<\ell+1$, let $\underline{v}_{ij} \in H_{\underline{n}_{ij}}$ be defined as the finest node that is split between districts $i$ and $j$.
The number of possible linking edges between the districts $i$ and $j$ may then be defined as
\begin{align}
\mathcal{L}_{ij} = \begin{cases}
|\{(u_i,u_j)\in E_0| u_i\in\xi_i, u_j \in \xi_j\}| \text{ if } \underline{n}_{ij} = \ell + 1\\
|\{(u_i,u_j)\in E_0| u_i\in\xi_i, u_j \in \xi_j, \text{ and } u_i,u_j\in\mathcal{Q}^{-\underline{n}_{ij}}(\underline{v}_{ij})\}| \text{ otherwise}
\end{cases}.
\end{align}
We remark that if two districts are not adjacent, then the number of linking edges is zero by this definition.
The total number of linking edge sets for a given parition $\xi$ may then be defined as
\begin{align}
\mathcal{L}(\xi) = \sum_{i,j} \max(1, \mathcal{L}_{ij}) = \sum_{\{i,j|\xi_i, \xi_j \text{ adjacent}\}} \mathcal{L}_{ij}.
\label{eq:linkededgecount}
\end{align}

Thus, this method of defining the linking edge sets is both robust and gives rise to a fairly simple computation for $\mathcal{L}(\xi)$ which is a product between border lengths of districts at (i) the finest split level (if the districts share a node in the graph hierarchy) or (ii) the entire border (otherwise).

\subsubsection{Strictly constrained splits}
\label{sssec:strictLinkedEdges}
There is one additional simplification we wish to make which naturally arises when strictly prescribing the number of split nodes at the highest level of the graph hierarchy.  Furthermore, although it is possible to use this additional simplification in a more general context, there may be complex combinatoric problems that render it either less efficient or infeasible (see Appendix~\ref{apdx:complxLinkEdges} for further discussion of the potential issues).

The main idea is that it may be desirable to sample the space of redistricting plans given an optimal set of coarse level node splits or a prescribed number coarse level node splits.  In this case, we may only choose to merge and split on a district that already splits a node at the highest level and must ensure that the resulting district pair also splits a node at the highest level.  

Because of this restriction, it is unnecessary to consider linking edges between districts that do not share some node of the graph hierarchy as we will never consider them.  This simplifies the computation of $\mathcal{L}(\xi)$ as we will no longer need to compute the border length between a wide number of districts.  Formally, when determining $\mathcal{L}(\xi)$, we can simply redefine 
\begin{align}
\mathcal{L}_{ij} = \begin{cases}
1 \text{ if } \underline{n}_{ij} = \ell + 1\\
|\{(u_i,u_j)\in E_0| u_i\in\xi_i, u_j \in \xi_j, \text{ and } u_i,u_j\in\mathcal{Q}^{-\underline{n}_{ij}}(\underline{v}_{ij})\}| \text{ otherwise}
\end{cases}.
\label{eq:linkededgecountFix}
\end{align}
and then again employ \eqref{eq:linkededgecount}.

\begin{remark}[Dynamic hierarchies]
We can still use these methods to sample without regard for fixed hierarchical elements by starting with an arbitrary hierarchy and then dynamically evolving the hierarchy in the Metropolis-Hastings context. Although we do not study this in the current work, we describe the method in further detail in Appendix~\ref{apdx:DynamicHierarchies}.
\end{remark}

\section{Numerical Results}
\label{sec:NumericalResults}

We implement our algorithm on a two level hierarchy of precincts and counties on the North Carolina congressional districts. The plan to be used in the 2020 North Carolina congressional election splits 12 counties, and does not split any county between more than two districts.  This is consistent with the 2016 redistricting criteria which required that no county be split into more than two districts \cite{congressionalCriteria2016}.  We adopt these constraints and (i) do not allow for any county to be split into more than two districts, and (ii) require there to be 13 or fewer split counties.

The one-person one-vote mandate requires that districts be made up of an equal number of people. We sample the districting space to examine plans that deviate up to 2\% away from the strictest interpretation of the mandate. We use the the 2010 census data for population data. Although a 2\% deviation is not strictly legally compliant, we have previously observed that such deviations do not effect observables of interest \cite{Herschlag20} and remark that the last several redistricting cycles split precincts to the census block level to achieve population parity. Reducing population deviations across districts may be done either by sampling from a three or more level hierarchy (e.g. county, precinct, census block), or by making small modifications to the sampled district plans by exchanging population along the boundaries.  In either case, we omit such explorations in the current work.

We restrict our analysis to $\gamma = 0$.  This restriction favors plans with a larger product of hierarchical trees. We remark that certain implementations of ReCom also focus on districts with larger numbers of spanning tree counts as moves are accepted so long as they are below a given compactness threshold \cite{deford2019recombination}.

In choosing a method of determining the linked edge set, we choose to only track 13 linked edges, only change an edge between two changed districts, and allow the edge to span any level of the hierarchy.
We do this because setting $\gamma = 0$ allows us to forgo counting the number of linking edges associated with a partition.  
The reason for fixing the number of linked edges is that we wish to gain insight into the mixing properties when we transition to using census blocks at the base level with stricter population constraints.  Under finer graphs with tighter constraints, we expect that merging and splitting two districts will rarely lead to counties (or coarsest nodes) that are kept intact. This means that the limitations on which linked edges to choose may become identical in both linking edge methods.  Under the current sampling process, which allows 2\% population deviations, it is possible that we will split fewer than 13 counties. If we employed linked edges over the entire space we would then be able to merge and split on any adjacent districts which should accelerate mixing.  Therefore our choice of linked edge method is \emph{more} restrictive, and will potentially slow mixing, but may be better reflective of future implementations of our algorithm.

We implement the multi-scale merge split algorithm on 10 independent chains.  Two of the chains are initialized with congressional districts used in 2016 and 2020.  On the 2020 map we introduce one more random linked edge between two adjacent districts that do not split a county.  The remaining eight chains are initialized on randomly seeded plans.  We construct these seed plans by drawing a tree on the county level, and search for nodes to cut so that one of the resulting cuts is within acceptable population bounds.  We draw a tree on a random cuttable node and then look for edges to remove that would lead to one of the cuts being within the population constraint.  We then repeat the above process on the remaining unassigned nodes until we obtain 13 districts within the population constraints. We illustrate the evolution of the chain starting at the 2016 congressional map after two and four acceptances and after 10,000 proposals in Figure~\ref{fig:chain}.

\begin{figure}
\centering
\subcaptionbox{Initial plan; the 2016 plan \label{sfig:2016initchain}}{\includegraphics[width=0.45\linewidth]{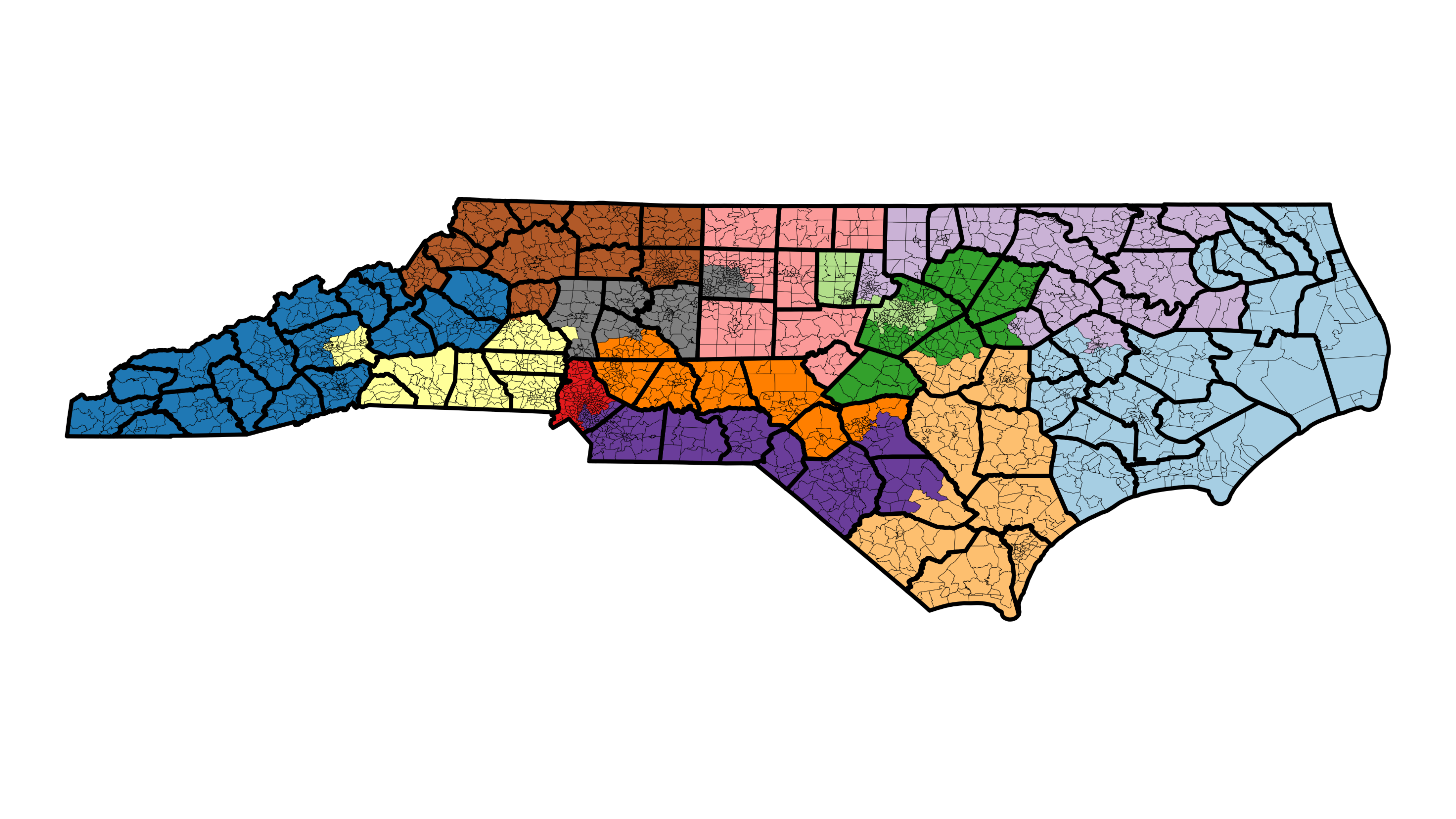}}\qquad
\subcaptionbox{State after two acceptances \label{sfig:2accept}}{\includegraphics[width=0.45\linewidth]{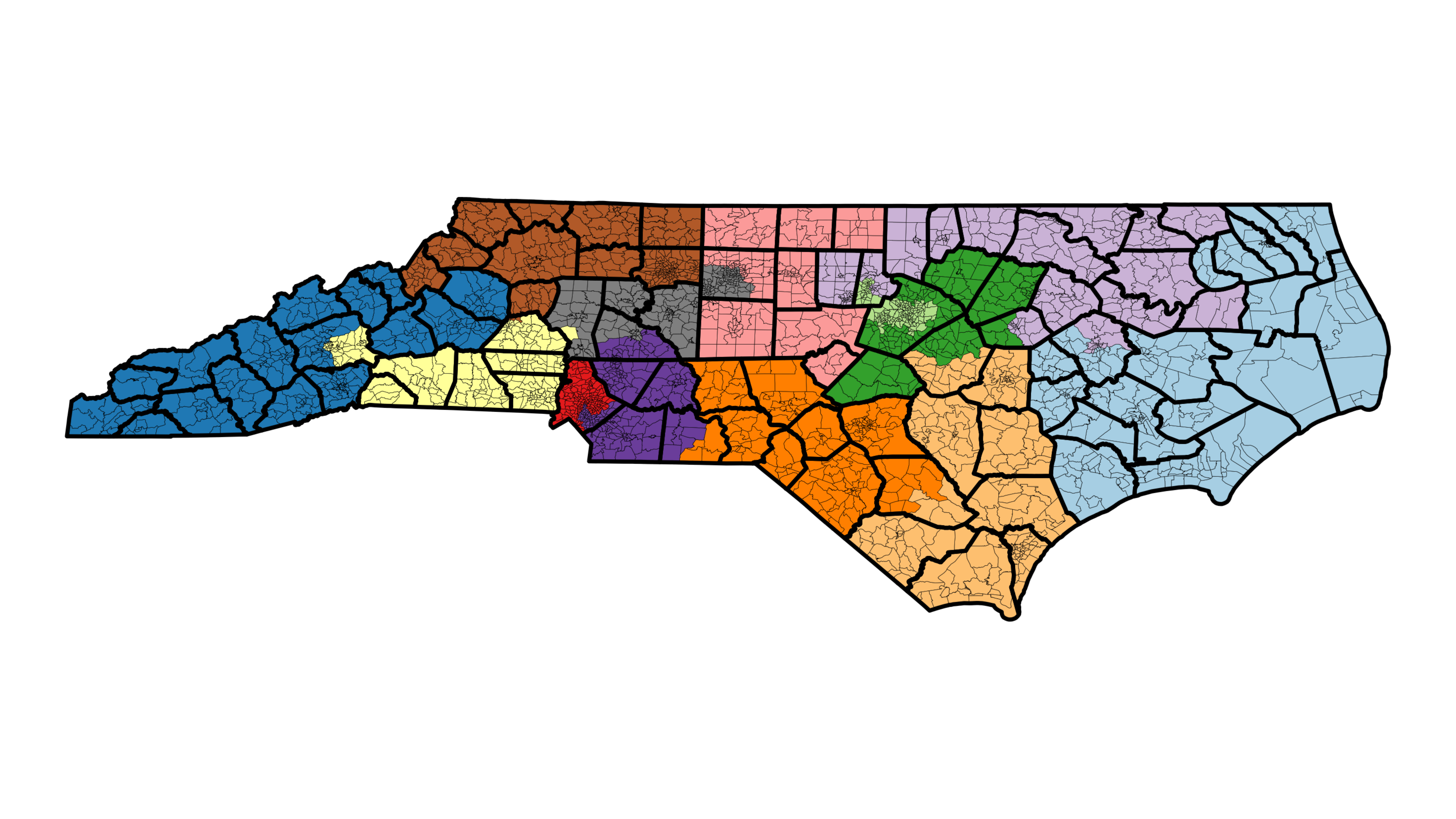}}\qquad
\subcaptionbox{State after four acceptances \label{sfig:4accept}}{\includegraphics[width=0.45\linewidth]{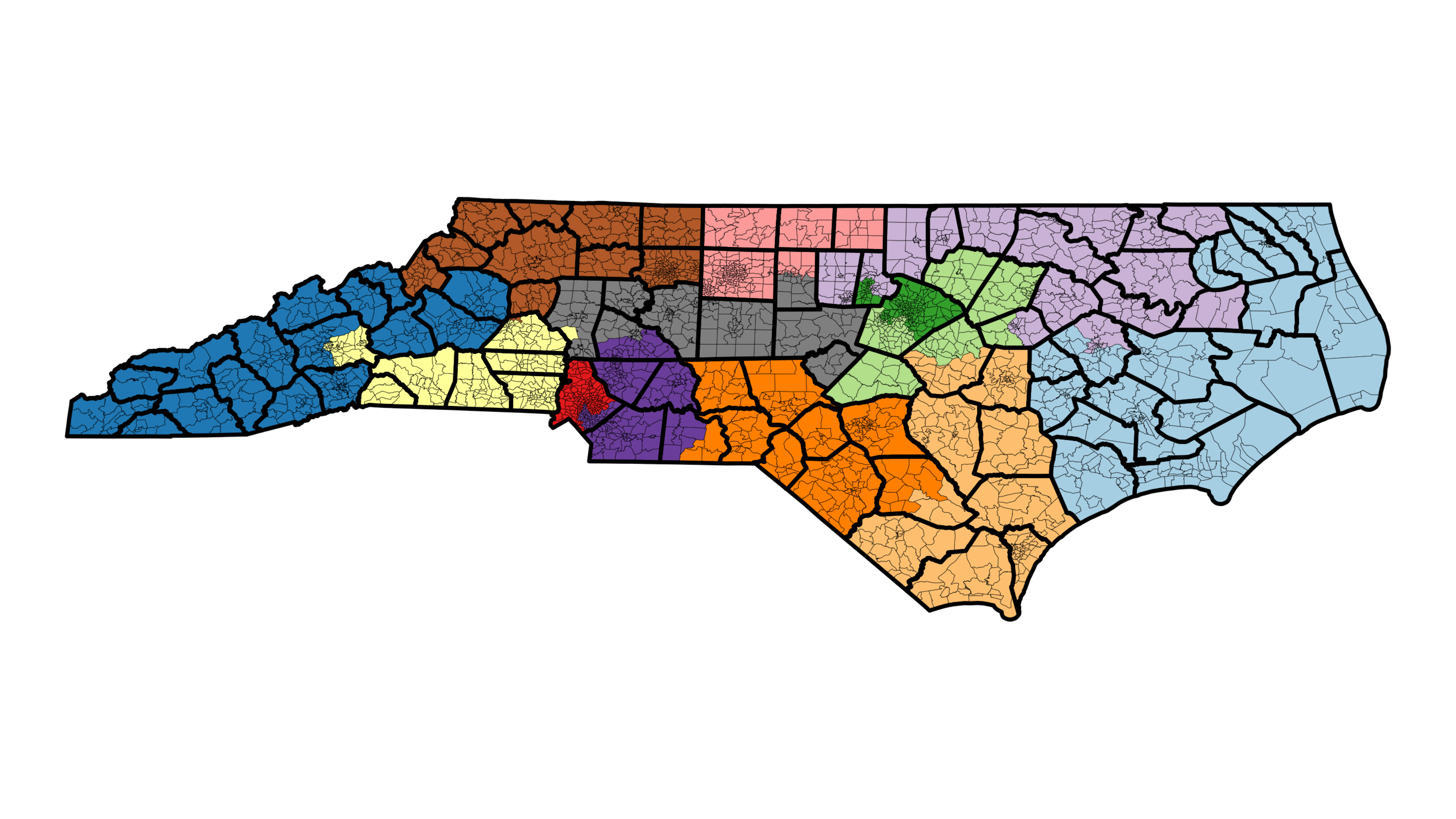}}\qquad
\subcaptionbox{State after roughly 2,500 acceptances \label{sfig:10Kprop}}{\includegraphics[width=0.45\linewidth]{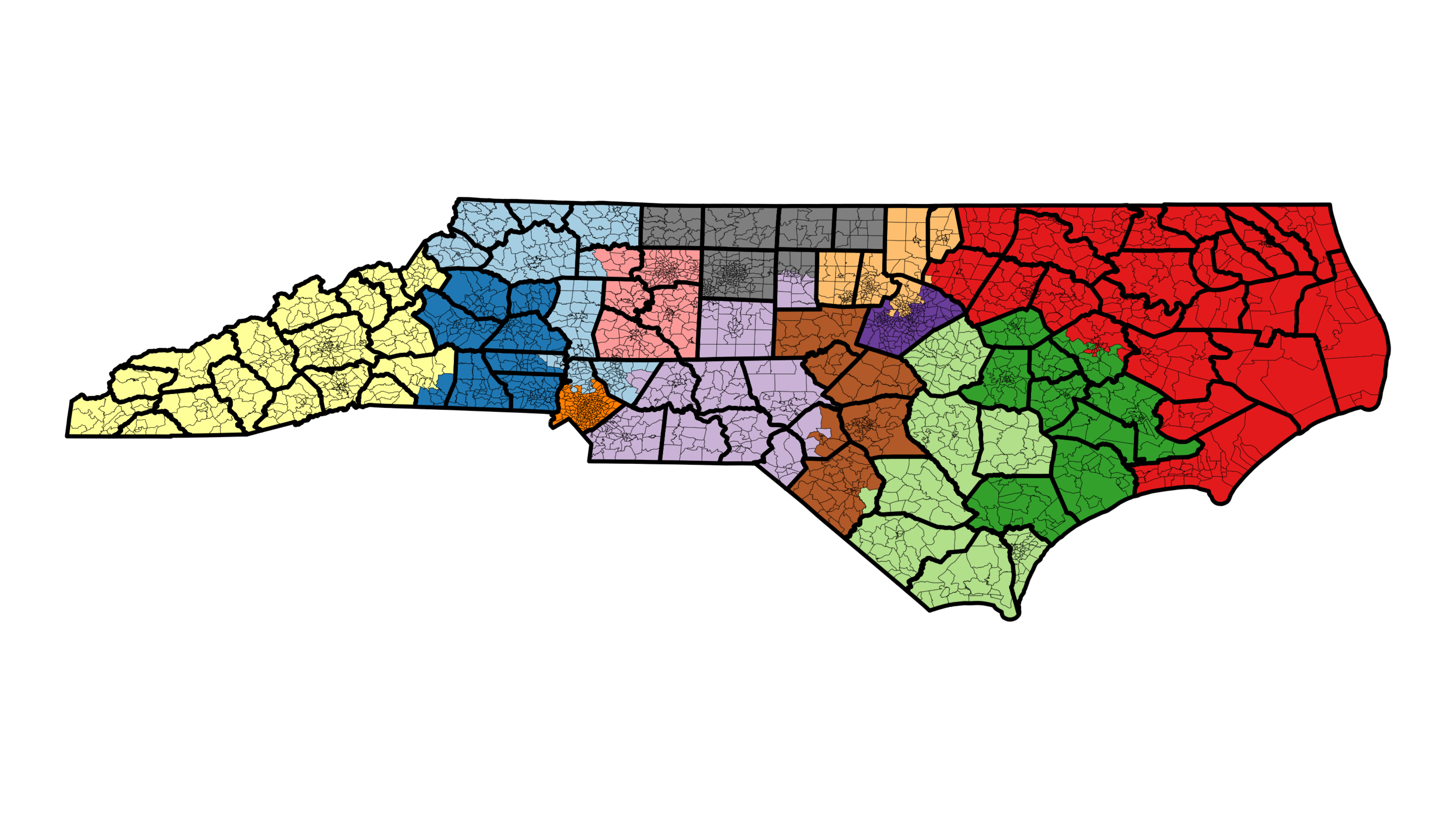}}
\caption{To illustrate how the multi-scale merge split chain may evolve, we show the initial condition of the 2016 congressional plan \ref{sfig:2016initchain} along with a realization of the  after two acceptances \ref{sfig:2accept}, four acceptances \ref{sfig:4accept} and 10,000 proposals (roughly 2,500 accepted steps) \ref{sfig:10Kprop}.}
\label{fig:chain}
\end{figure}

For each sample we count the number of Democratic and Republican votes within each district according to the 2012 presidential vote counts.  We then examine the fraction of the Republican vote compared to the total Democratic and Republican votes.  For each plan in the ensemble, we count the number of districts that have more Democratic votes than Republican votes, which serves as a proxy for the number of Democratic congressional delegates that would have been elected under these votes.  We construct a histogram of elected Democrats for each chain, and then examine the total variation between the histograms across any two chains.  We report the largest total variation in the histograms of all chains; we examine (i) this maximal pairwise variation as a function of the number of proposals, (ii) the histograms between the two chains with the highest variation, and (iii) the histogram produced over all chains (see Figure~\ref{fig:hists}). We find that after $8\times 10^5$ proposals, the largest total variation between chains is 3.56\% and that the chains are converging with an order of $0.56$.  We also find that the largest total variation between any given chain and all chains is 2.46\%.  In short, we find evidence that all 10 chains are converging to a stationary distribution.

\begin{figure}
\centering
\includegraphics[height=4.5cm]{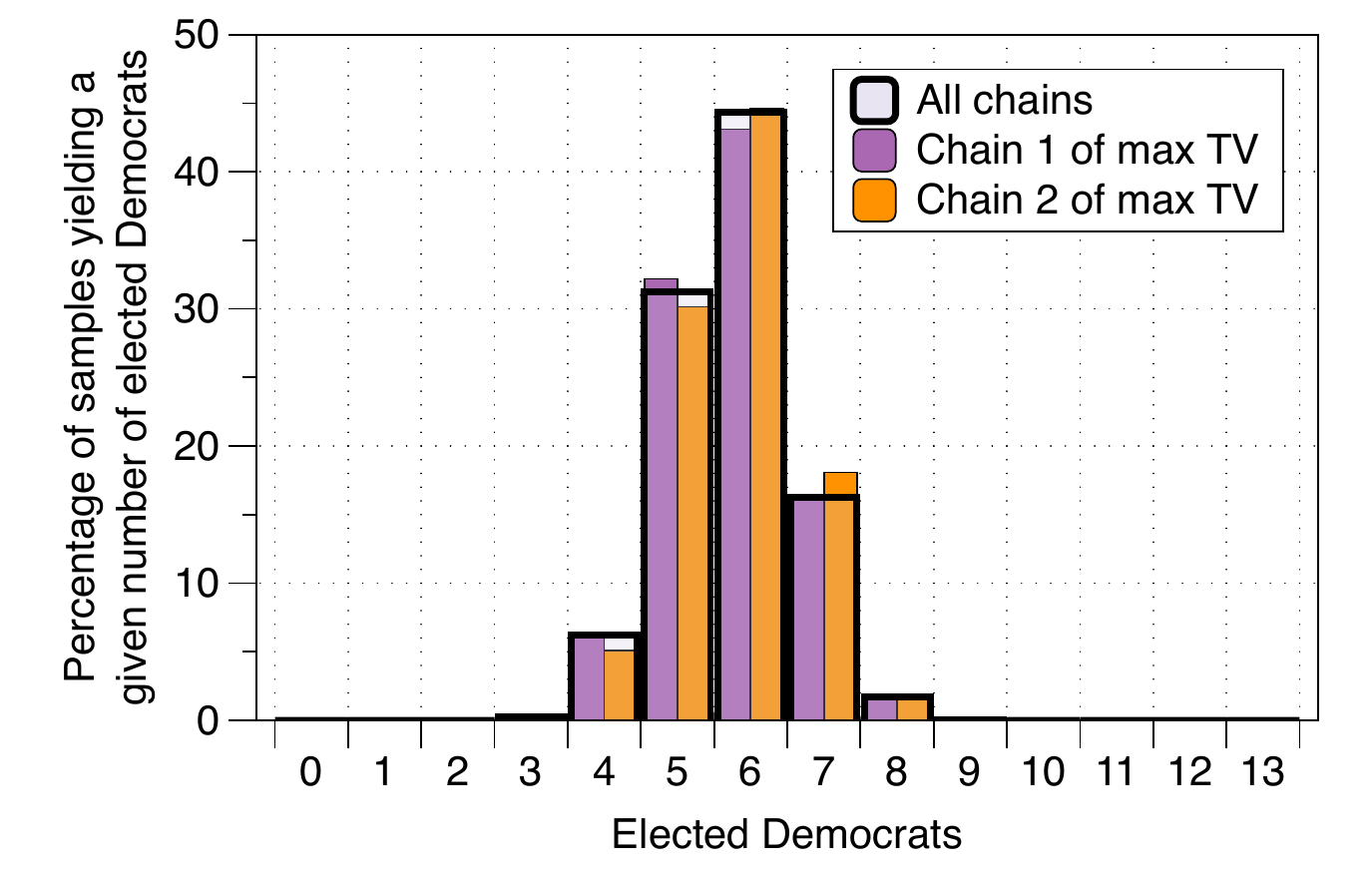}\qquad
\includegraphics[height=4.5cm]{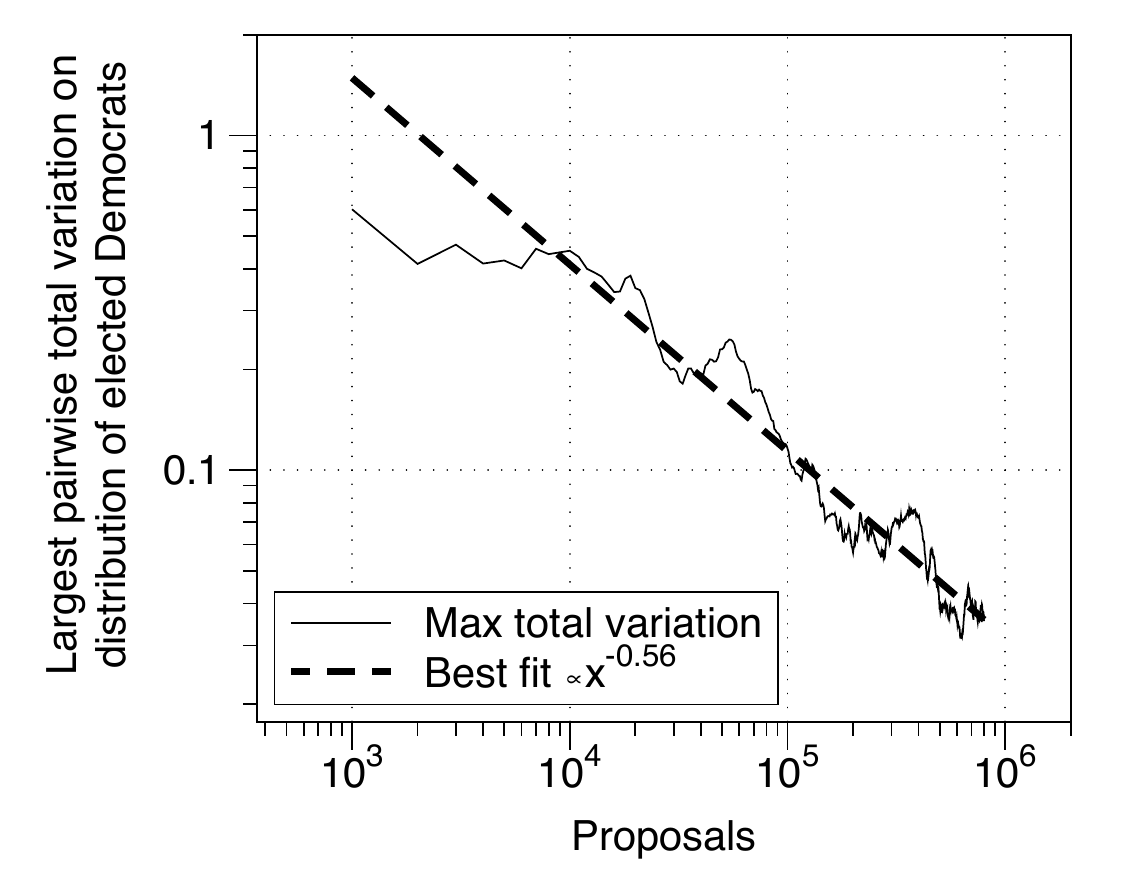}
\caption{We plot the histograms of the number of elected Democrats under the 2012 Presidential vote. We plot the histograms of the two chains with the largest total variation between them and overlay the histogram determined over all chains (left).  We plot the largest pairwise total variation as a function of the number of proposals, and find and order of convergence of roughly one half (right).}
\label{fig:hists}
\end{figure}

We next order the thirteen districts from most to least Republican. We use the ensemble of plans to examine the marginal distributions of the most Republican plan, the second most Republican plan, and so on until we arrive at the most Democratic plan. To study the convergence of these order statistics, we generate histograms of bin size 0.2\% for each marginal ensemble within each chain as a function of the number of proposals.  
We then examine the total variation of all ordered marginal distributions between two chains, and then average these variations.
We report the pair of chains with largest average total variation as a function of the number of proposals.
We plot the largest variation as a function of the number of proposals as well as examine the resulting rank-ordered marginal histograms after 800,000 proposals in Figure~\ref{fig:marginals}.  We find an estimated rate of convergence of $0.42$.  Qualitatively, the marginal distributions between the chains are quite close to one another, however there still a few structural differences as can be seen in the marginal distribution of the most Democratic district.  Nevertheless, we see evidence that the chains are converging to the same distribution. 

\begin{figure}
\centering
\includegraphics[height=4.5cm]{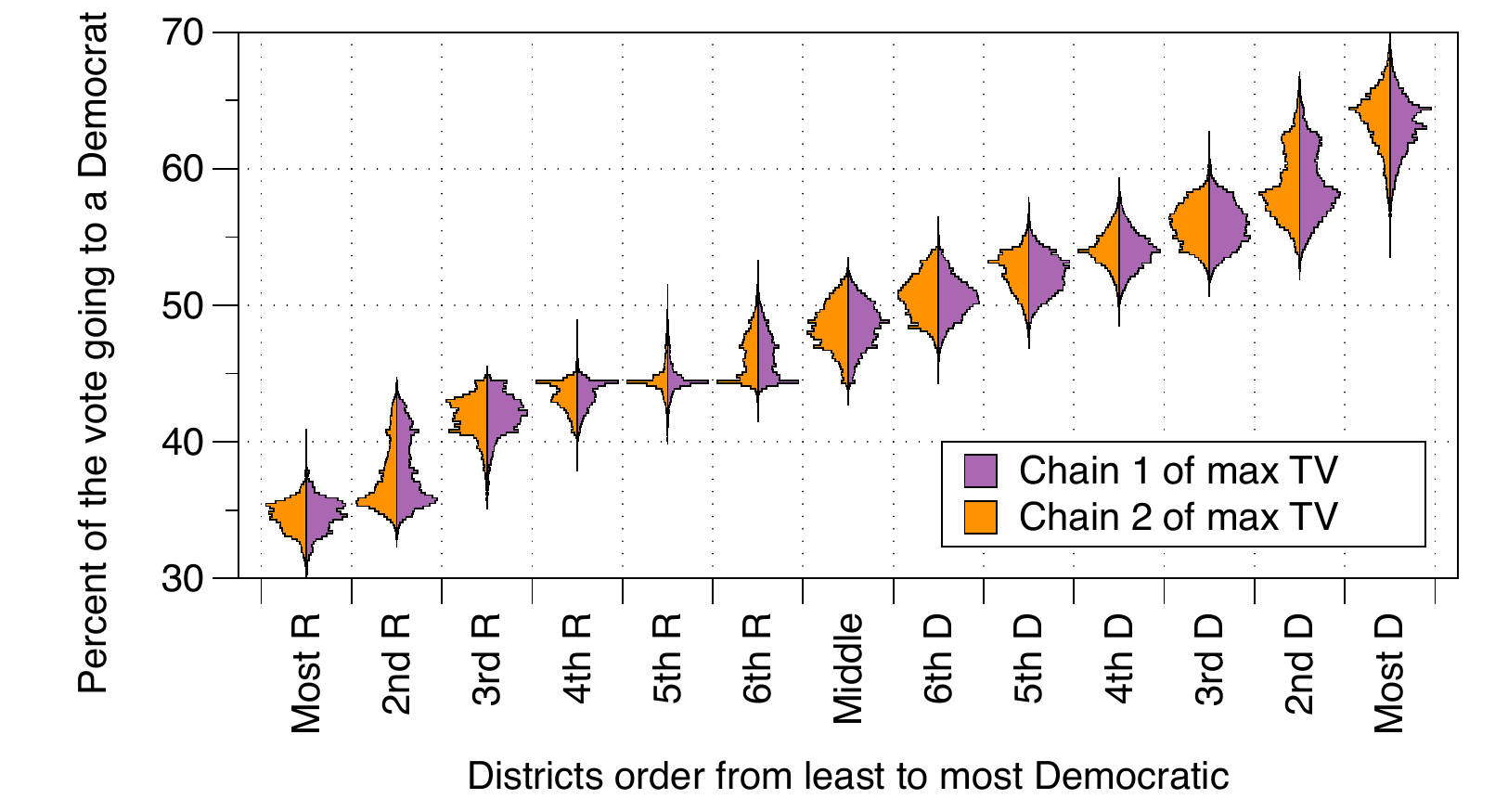}\qquad
\includegraphics[height=4.5cm]{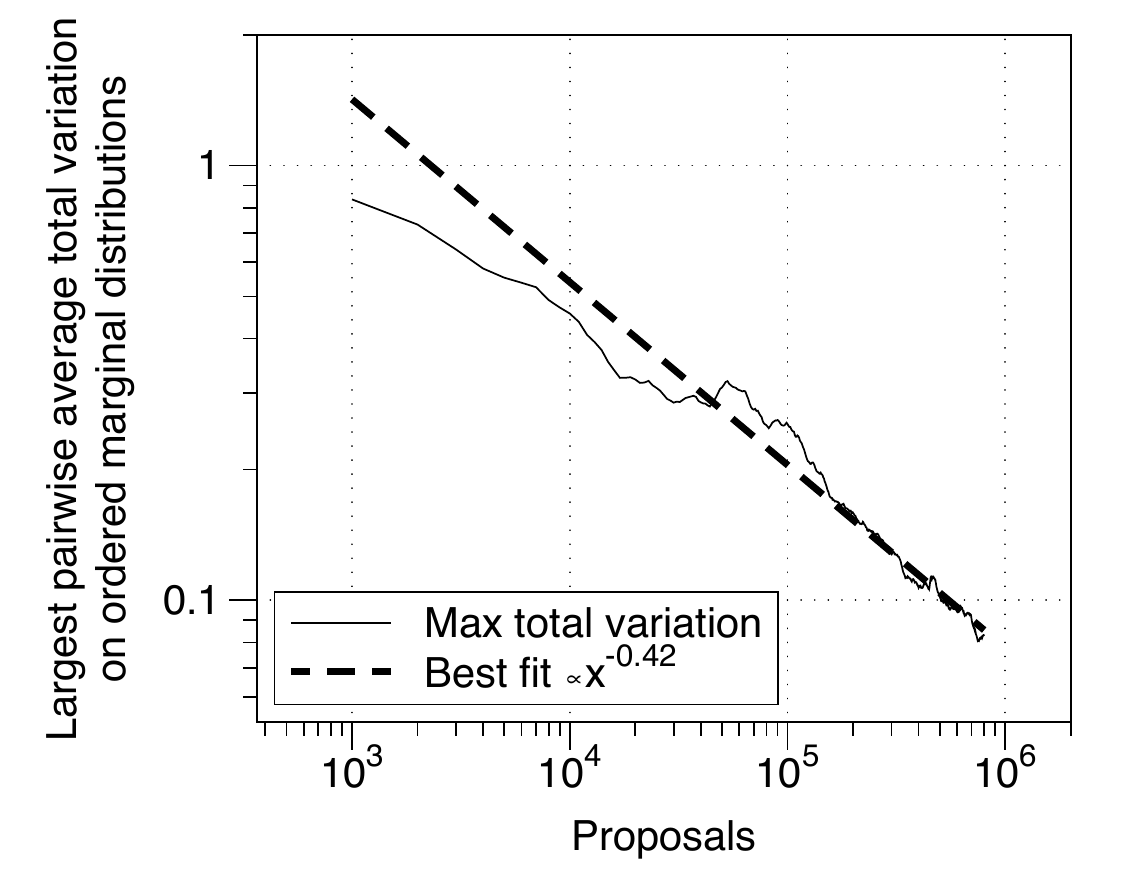}
\caption{We plot the rank-ordered marginal histograms of the percentage of the Democratic vote in each district under the 2012 Presidential vote. The districts are ordered from the most Republican (Most R) to most Democratic (Most D). We plot the marginal distributions from the two chains with the largest total variation averaged across the 13 districts after 800,000 proposals (left).  We plot the largest pairwise average total variation as a function of the number of proposals and find and order of convergence of roughly one $0.42$ (right).}
\label{fig:marginals}
\end{figure}

\section{Discussion}
We have developed a Multi-Scale Merge-Split algorithm to sample the space of redistricting plans. This algorithm builds on our pervious Merge-Split algorithm which uses spanning trees rather than hierarchical trees \cite{carter2019merge}. In turn, Merge-Split builds on the ReCom algorithm \cite{deford2019recombination} by using a tree based variant of ReCom as its proposal. The algorithm reduces the computational complexity of counting trees on district graphs by drawing and counting hierarchical trees rather than spanning trees. Given a graph with $N$ nodes, counting the number of hierarchical trees in a multi-scale setting has a complexity of $O(\log(N))$ compared to counting spanning trees which has a complexity of $O(N^{2.373})$.

The algorithm may use a specified hierarchy to preserve geographic regions of interests such as counties and precincts, or may use a dynamic hierarchy which will retain the scaling benefits while sampling a less constrained space of redistricting plans. In this initial report, we have demonstrated our algorithm on a hierarchy of precincts and counties and have demonstrated promising convergence results. We have not yet implemented dynamic hierarchies nor have we gone down to the census block level, but we expect this method will perform well in both of these cases, and plan to explore this in a future update.

In addition to the implementations we have discussed, there are a number of ways to extend the ideas in this work. For instance, we have only considered situations in which coarse nodes may only be split into two districts, however this method may be easily extended to cases in which nodes are split across more than two districts.  We have also only considered districting plans that share at most one node at any given level, however it should be possible to extend this algorithm to allow for multiple split coarse nodes across two districts.  

There are also open questions about mixing. For instance, what (if any) constraints will not allow the multi-scale merge split algorithm to mix? Perhaps there are graphs and county splitting constraints that will prevent the algorithm from being ergodic. Despite these open questions, the multi-scale algorithm promises to sample on finely resolved redistricting plans, perhaps even down to the census block level, while sampling from a known invariant measure.

\subsection*{Acknowledgements:} This work started as part of the ``Discovering Research Mathematic'' program between Duke Math and NCSSM in which Dan Teague played a critical role.  JCM and GJH thank the workshop ``Redistricting 2020'' in March of 2020 sponsored by the Duke NSF TRIPODS grant (NSF-CFF-1934964) and the Duke Rhodes Information Initiative. JCM and GJH thank Andrea Agazzi for useful discussion.  JCM also thanks the NSF grant 1613337 for partial support.

\bibliography{./biblio}

\newcommand{\etalchar}[1]{$^{#1}$}
\begin{thebibliography}{CHHM19}

\bibitem[AH74]{aho1974design}
Alfred~V Aho and John~E Hopcroft.
\newblock {\em The design and analysis of computer algorithms}.
\newblock Pearson Education India, 1974.

\bibitem[BDGV15]{QuantifyingGerrymandering}
Sachet Bangia, Bridget Dou, Sophie Guo, and Christy Vaughn.
\newblock {Q}uantifying {G}errymandering {D}ata+ project.
\newblock https://services.math.duke.edu/projects/gerrymandering/, 2015.

\bibitem[BGH{\etalchar{+}}17]{Bangia17}
Sachet Bangia, Christy~Vaughn Graves, Gregory Herschlag, Han~Sung Kang, Justin
  Luo, Jonathan~C. Mattingly, and Robert Ravier.
\newblock Redistricting: Drawing the line.
\newblock {\em arXiv}, 1704.03360, 2017.

\bibitem[Can20]{CannonConference2020}
Sarah Cannon.
\newblock Recombination, reversibility, and short bursts.
\newblock https://bigdata.duke.edu/tripods/redistconf20, 2020.

\bibitem[CDO00]{cirincione2000assessing}
Carmen Cirincione, Thomas~A Darling, and Timothy~G O'Rourke.
\newblock Assessing south carolina's 1990s congressional districting.
\newblock {\em Political Geography}, 19(2):189--211, 2000.

\bibitem[CFMP19]{chikina2019separating}
Maria Chikina, Alan Frieze, Jonathan Mattingly, and Wesley Pegden.
\newblock Separating effect from significance in {M}arkov chain tests, 2019.

\bibitem[CFP17]{Chikina_Frieze_Pegden_2017}
Maria Chikina, Alan Frieze, and Wesley Pegden.
\newblock Assessing significance in a {M}arkov chain without mixing.
\newblock {\em Proceedings of the National Academy of Sciences},
  114(11):2860--2864, Mar 2017.

\bibitem[CHHM19]{carter2019merge}
Daniel Carter, Gregory Herschlag, Zach Hunter, and Jonathan Mattingly.
\newblock A merge-split proposal for reversible monte carlo markov chain
  sampling of redistricting plans.
\newblock {\em arXiv preprint arXiv:1911.01503}, 2019.

\bibitem[CHT{\etalchar{+}}19]{carter2019optimal}
Daniel Carter, Zach Hunter, Dan Teague, Gregory Herschlag, and Jonathan
  Mattingly.
\newblock Optimal legislative county clustering in {N}orth {C}arolina, 2019.

\bibitem[Cov]{CovingtonvNC}
{\em \textup{Covington v. North Carolina, No. 16-1023 (2016)}}.

\bibitem[CR13]{ChenRodden13}
Jowei Chen and Jonathan Rodden.
\newblock Unintentional gerrymandering: Political geography and electoral bias
  in legislatures.
\newblock {\em Quarterly Journal of Political Science}, 8:239--269, 2013.

\bibitem[CR15]{Chen15}
Jowei Chen and Jonathan Rodden.
\newblock Cutting through the thicket: Redistricting simulations and the
  detection of partisan gerrymanders.
\newblock {\em Election Law Journal}, 14(4):331--345, 2015.

\bibitem[DD19a]{DeFordDuchinPrivite}
Daryl DeFord and Moon Duchin.
\newblock Privite communications, 2018,2019.

\bibitem[DD19b]{deford2019redistricting}
Daryl DeFord and Moon Duchin.
\newblock Redistricting reform in virginia: Districting criteria in context.
\newblock {\em Preprint}, 1:22, 2019.

\bibitem[DDS19]{deford2019recombination}
Daryl DeFord, Moon Duchin, and Justin Solomon.
\newblock Recombination: A family of markov chains for redistricting.
\newblock {\em arXiv preprint arXiv:1911.05725}, 2019.

\bibitem[DeF18]{DeFord2018}
{\em Compactness Profiles and Reversible Sampling Methods for Plane and Graph
  Partitions}, Quantitative Redistricting Conffernece, October 2018.
\newblock Conference Talk at
  https://www.samsi.info/programs-and-activities/research-workshops/quantitative-redistricting/.

\bibitem[Duc]{duchinPAreport}
Moon Duchin.
\newblock Outlier analysis for pennsylvania congressional redistricting.
\newblock {\em
  https://www.governor.pa.gov/wp-content/uploads/2018/02/md-report.pdf}.

\bibitem[FHIT15]{fifield2015}
Benjamin Fifield, Michael Higgins, Kosuke Imai, and Alexander Tarr.
\newblock A new automated redistricting simulator using {M}arkov chain {M}onte
  {C}arlo.
\newblock {\em Work. Pap., Princeton Univ., Princeton, NJ}, 2015.

\bibitem[GIKM17]{greenhillAverageNumberSpanning2017}
Catherine Greenhill, Mikhail Isaev, Matthew Kwan, and Brendan~D. McKay.
\newblock The average number of spanning trees in sparse graphs with given
  degrees.
\newblock {\em European Journal of Combinatorics}, 63:6--25, February 2017.

\bibitem[Gil]{GillVWhitford}
{\em \textup{Gill v. Whitford, No. 16-1161, 585 U.S. \_\_\_ (2018)}}.

\bibitem[Gre]{GreensboroVGuilford}
{\em \textup{City of Greensboro et al. v. Guilford County Board of Elections
  (M.D.N.C. 2015)}}.

\bibitem[HB119]{HB1029}
An act to realign the congressional districts.
\newblock https://www.ncleg.gov/BillLookUp/2019/H1029, 2019.

\bibitem[HKL{\etalchar{+}}20]{Herschlag20}
Gregory Herschlag, Han~Sung Kang, Justin Luo, Sachet Bangia, Christy~Vaughn
  Graves, Robert Ravier, and Jonathan~C. Mattingly.
\newblock Quantifying gerrymandering in north carolina.
\newblock {\em Statistics and Public Policy}, Under review, 2020.

\bibitem[HMnt]{QuantifyingGerrymanderingBlog}
Gregory Herschlag and Jonathan Mattinlgly.
\newblock {Q}uantifying {G}errymandering blog.
\newblock https://sites.duke.edu/quantifyinggerrymandering/, 2018 -- Present.

\bibitem[HRM17]{herschlag2017evaluating}
Gregory Herschlag, Robert Ravier, and Jonathan~C. Mattingly.
\newblock Evaluating partisan gerrymandering in wisconsin, 2017.

\bibitem[LCW16]{Liu16}
Yan~Y. Liu, Wendy K.~Tam Cho, and Shaowen Wang.
\newblock Pear: a massively parallel evolutionary computation approach for
  political redistricting optimization and analysis.
\newblock {\em Swarm and Evolutionary Computation}, 30:78--92, 2016.

\bibitem[Leg16]{congressionalCriteria2016}
NC~Legislature.
\newblock 2016 contingent congressional plan committee adopted criteria.
\newblock See
  https://www.ncleg.gov/Files/GIS/ReferenceDocs/2016/CCP16\_Adopted\_Criteria.pdf,
  2016.

\bibitem[Lew]{LewisVCommonCause}
{\em \textup{Common Cause v. Lewis, (2019)}}.

\bibitem[LWV]{LWVvPA}
{\em \textup{League of Women Voters of Pennsylvania v. Commonwealth of
  Pennsylvania, No. 159 MM (2017)}}.

\bibitem[Mac01]{macmillan2001redistricting}
William Macmillan.
\newblock Redistricting in a gis environment: An optimisation algorithm using
  switching-points.
\newblock {\em Journal of geographical systems}, 3(2):167--180, 2001.

\bibitem[Mat19]{jcmReport}
Jonathan~C. Mattingly.
\newblock Expert report for {C}ommon {C}ause v. {L}ewis.
\newblock {\em {C}ommon {C}ause v. {L}ewis}, 2019.

\bibitem[MGG]{moonVa}
MGGG.
\newblock Comparison of districting plans for the virginia house of delegates.
\newblock {\em https://mggg.org/ VA-report.pdf}.

\bibitem[MJN98]{mehrotra1998optimization}
Anuj Mehrotra, Ellis~L Johnson, and George~L Nemhauser.
\newblock An optimization based heuristic for political districting.
\newblock {\em Management Science}, 44(8):1100--1114, 1998.

\bibitem[MV14]{MattinglyVaughn2014}
J.~C. {Mattingly} and C.~{Vaughn}.
\newblock {Redistricting and the Will of the People}.
\newblock {\em ArXiv e-prints}, October 2014.

\bibitem[NDS19]{njatDedfordSolomon2019graphs}
Lorenzo Najt, Daryl Deford, and Justin Solomon.
\newblock Complexity and geometry of sampling connected graph partitions.
\newblock {\em Preprint; arxiv.org/pdf/1908.08881.pdf}, 2019.

\bibitem[Ruc]{RuchoVCC}
{\em \textup{Rucho v. Common Cause, No. 18-422, 588 U.S. \_\_\_ (2019)}}.

\bibitem[RWC]{RWCAvWakeBOE}
{\em \textup{Raleigh Wake Citizens Association et al. v. Wake County Board of
  Elections (E.D.N.C. 2015)}}.

\bibitem[WDS{\etalchar{+}}15]{Wu15}
Lucy~Chenyun Wu, Jason~Xiaotian Dou, Danny Sleator, Alan Frieze, and David
  Miller.
\newblock Impartial redistricting: A {M}arkov chain approach.
\newblock {\em arXiv:1510.03247v1}, 2015.

\end{thebibliography}
\bibliographystyle{alpha}

\appendix

\section{Dynamic Hierarchies}
\label{apdx:DynamicHierarchies}
In the current work we have focused on and numerically tested hierarchies on existing geographic units such as precincts embedded in counties. However, the multi-scale merge split algorithm need not be confined to prescribed hierarchies and may instead apply to more general and less constrained redistricting problems.

In this section we add a brief discussion of how this would work. First, one would decide on a number of hierarchies, $\ell$, on a base graph with $N$ nodes. This would yield, on average, $m = \sqrt[\ell]{N}$ nodes per partition within the quotient graphs. One could then initialize $m$ partitions on the fine scale graph by drawing a tree on the base graph and looking to cut it to separate out $m$ vertices (or alternatively a given fraction of the population), and then recurse with the remaining part of the graph that has not yet been assigned to a partition. Such a procedure would generate some randomized hierarchical graph.

If we were to sample on this fixed hierarchy, our sampling space would be constrained by which partitions could be simultaneously cut. We may, however, dynamically alter the graph hierarchy. To do this, we expand the measure to be a joint distribution on the hierarchy and the expanded districting (trees and liked edges). The form of the distribution is precisely the same:
\begin{align}
P(T, L, \mathcal{H}) \propto e^{-\beta J(\xi(\allT))} \big(\tau_\mathcal{H}(\xi(\allT))\times \mathcal{L}_{\mathcal{H}}(\xi(\allT))\big)^{-\gamma},
\end{align}
where we have now made it explicit that the number of linked edges depends on the hierarchy.

We then have two types of proposals -- one to evolve the districts on a fixed hierarchy and another to evolve the hierarchy on fixed districts. We have already developed and tested the first in the main text. The second would work as follow:
\begin{enumerate}
\item Choose an arbitrary level, $l$, of the hierarchy (larger than the base level). 
\item Choose an arbitrary pair of adjacent nodes, $u$ and $v$, at level $l$ that are (i) wholly contained in the same district, (ii) that are connected by an edge one of the trees, $T_i$, in the forest $T$, and (iii) belong to the same partition at level $l+1$.
\item If not already resolved, resolve the spanning trees within the nodes at the next finer level, $l-1$, by drawing a uniform spanning tree on the multi-graph of the this level ($l-1$). Call these trees $t_u$ and $t_v$, and the edge linking them $e_{(u,v)}$.
\item Draw a new uniform spanning tree on the finer graph induced by $u$ and $v$ (denoted $\mathcal{Q}^{-1}(u\cup v)$). Look for edges (and only edges) that could be removed on the merged tree such that the two resulting trees would contain an acceptable number of nodes. If there are no such edges, reject the move.
\item Pick one of the edges to remove, $e_{(u',v')}$, and form two new trees, $t_u'$ and $t_v'$ at level $l-1$, which will correspond to two new nodes, $u'$ and $v'$ at level $l$.
\item Use the removed edge as the new edge which connects the updated tree, $T_i'$. Note that the district has not been altered even though the tree representing the district has been altered, i.e. $\xi(T_i') = \xi(T_i)$. Note also that the edges linking the altered nodes to the unaltered nodes remain in place and that the overall topology of the new district tree $T_i'$ may change.
\item Compute the probability of proposing the new nodes and hierarchy along with the reverse probability. Use this to compute the acceptance ratio.
\end{enumerate} 

This algorithm operates almost identically to the original Merge-Split algorithm, but the measure is modified, and we are merging and splitting on sub-elements of the district trees, rather than on a pair of district trees.

The forward and backward proposal probabilities are computed similarly to the Merge-Split algorithm as
\begin{align}
Q((T, L, \mathcal{H}), (T', L, \mathcal{H}')) = P(\text{level $l$}) P(u,v | l, T) \frac{1}{\tau(\mathcal{Q}^{-1}(u\cup v))} P_cut(e_{(u',v')} | T_{(t_u', t_v', e_{(u',v')})}),
\end{align}
where $P(\text{level $l$})$ is the chance of choosing level $l$, and $P(u,v | l, T)$ is the chance of picking two nodes $u$ and $v$ at level $l$ that are linked through $T$, and within the same larger partition (if $l<\ell$). The third fraction on the right hand side is the number of spanning trees on the finer level graph induced by the two partitions given by $u$ and $v$. The final term is the probability of cutting a given edge given the uniformly drawn spanning tree $T_{(t_u', t_v', e_{(u',v')})}$.

If we assume that the probability of picking a certain level is independent of the state, and that the probability of picking a two adjacent nodes linked through $T$ is uniform, then the ratio of proposal probabilities simplifies to 
\begin{align}
\frac{Q((T, L, \mathcal{H}), (T', L, \mathcal{H}'))}{Q((T', L, \mathcal{H}'), (T, L, \mathcal{H}))} = \frac{P_cut(e_{(u',v')} | T_{(t_u', t_v', e_{(u',v')})})}{P_cut(e_{(u,v)} | T_{(t_u, t_v, e_{(u,v)})})},
\end{align}
since the number of adjacent nodes that we can choose from remain fixed between $T$ and $T'$.

To complete the algorithm, we must compute the acceptance ratio, which is given by 
\begin{align} 
A((T, L, \mathcal{H}), (T', L, \mathcal{H}')) = \min\Bigg(1, \frac{Q((T', L, \mathcal{H}'), (T, L, \mathcal{H}))}{Q((T, L, \mathcal{H}), (T', L, \mathcal{H}'))} \frac{P(T', L, \mathcal{H}')}{P(T, L, \mathcal{H})}\Bigg).
\end{align} 
The probability ratio simplifies to
\begin{align}
\frac{P(T', L, \mathcal{H}')}{P(T, L, \mathcal{H})} = \frac{\tau_\mathcal{H'}(\xi(\allT'))^{-\gamma}}{\tau_\mathcal{H}(\xi(\allT))^{-\gamma}} = \Bigg(
\frac{
\tau(\mathcal{Q}^{-1}_{\mathcal{H'}}(u')) \tau(\mathcal{Q}^{-1}_{\mathcal{H'}}(v')) \tau(\mathcal{Q}^{-1}_{\mathcal{H'}}(\mathcal{Q}_{\mathcal{H}}(u))) 
}{
\tau(\mathcal{Q}_{\mathcal{H}}^{-1}(u')) \tau(\mathcal{Q}_{\mathcal{H}}^{-1}(v')) \tau(\mathcal{Q}^{-1}_{\mathcal{H}}(\mathcal{Q}(u)))
}
\Bigg)^{-\gamma},
\end{align}
where $\mathcal{Q}$ is the quotient (coarsening) operator and $\mathcal{Q}^{-1}$ expands a coarse graph to the finer level of the hierarchy; in this case we specify which hierarchy the quotient operator is associated with by adding the subscripts $\mathcal{H}$ or $\mathcal{H'}$.

\section{Potential computational complexities with a fixed number of linking edges}
\label{apdx:complxLinkEdges}
As mentioned in Section~\ref{ssec:linkingEdgeSets}, when a linking edge is strictly contained within a partition at the coarsest level (i.e. a single node in $H_{\ell}$), any edge linking the two corresponding districts will also be confined to this partition. At times we may even require that linking edges are confined to coarse partitions (see Section~\ref{sssec:strictLinkedEdges}).  Counting the number of linking edge sets is straight forward in this case as (i) there must be a linked edge between the two districts sharing a coarse node and (ii) that linked edge must be entirely contained within the shared coarse node (see \eqref{eq:linkededgecount} and \eqref{eq:linkededgecountFix} in Section~\ref{ssec:linkingEdgeSets}). 

Suppose that we wish to prescribe a fixed number of linked edges, as in Section~\ref{sssec:strictLinkedEdges}, but we wish to allow edges to span two nodes at the coarsest level (in $H_{\ell}$).  In computing total number of possible linking edge sets that corresponds to this state, we would have to consider edges linking coarse nodes across \emph{any} adjacent districts that are not split across a coarse node.  Dealing with one such global linking edge is difficult, but dealing with multiple quickly becomes intractable. 

For example, consider the 2016 congressional map showing in Figure~\ref{sfig:2016initchain} containing 13 split counties.  Consider a linking edge scheme that fixes more than 14 linking edges. Thirteen of the linking edges will reside entirely within split counties, but remaining edges are far less constrained:  We could add one that links the dark brown to dark blue western districts, one that links the central light red district to the light green district, or one that links the eastern light blue district to the light orange district.  

If we were only adding one more linking edge, we would have to count the number of possible links we could make, but two more edges become more challenging.  In counting adding the number of ways two add two edges, we would have to count the number of ways to make the second link given all possible choices for the first link.  Given three more linking edges, we would count the ways to add a third edge for all possible choices of choosing the first two edges.  In all cases, we would have to avoid counting linking edge sets more than once.

This combinatoric complexity is what makes it more difficult to count fixed-size linking edge sets associated with a partition when the number of linking edges exceeds the number of split counties.  In the numerical example we have examined in this paper, we set $\gamma=0$ which enabled us to avoid counting possible linking edge sets and were therefore free to have these linking edges span counties (i.e. coarse nodes).

\end{document}